\documentclass{article}
\usepackage{amsmath}
\usepackage{amssymb}
\usepackage{amsfonts}

\newtheorem{theorem}{Theorem}[section]

\newtheorem{case}{Case}
\newtheorem{claim}{Claim}[section]

\newtheorem{condition}{Condition}[section]

\newtheorem{corollary}{Corollary}

\newtheorem{definition}{Definition}[section]
\newtheorem{example}{Example}

\newtheorem{lemma}{Lemma}[section]

\newtheorem{proposition}{Proposition}[section]

\newenvironment{proof}[1][Proof]{\noindent\textbf{#1.} }{\ \rule{0.5em}{0.5em}}

\input{tcilatex}
\begin{document}

\title{IBN-varieties of algebras.}
\author{A. Tsurkov \\
Mathematical Department, CCET,\\
Federal University of Rio Grande do Norte (UFRN),\\
Av. Senador Salgado Filho, 3000,\\
Campus Universit\'{a}rio, Lagoa Nova, \\
Natal - RN - Brazil - CEP 59078-970;\\
Mathematical Department,\\
Institute of Mathematics and Statistics,\\
University of S\~{a}o Paulo, Rua do Mat\~{a}o, 1010,\\
S\~{a}o Paulo - SP - Brazil - CEP 05508-090;\\
arkady.tsurkov@gmail.com}
\maketitle

\begin{abstract}
The concept of a variety with IBN (invariant basic number) propriety first
appeared in the ring theory. But we can define this concept for an arbitrary
variety $\Theta $ of universal algebras with an arbitrary signature $\Omega $%
; see Definition \ref{IBN-variety}.

The proving of the IBN propriety of some variety is very important in
universal algebraic geometry. This is a milestone in the study of the
relation between geometric and automorphic equivalences of algebras of this
variety.

In this paper we prove very simple but very useful for studying of IBN
proprieties of different varieties Theorems \ref{terms} and \ref{main}. We
will consider some applications of this theorem.

We will consider many-sorted universal algebras as well as one-sorted. So
all concepts and all results will by generalized for the many-sorted
case.\bigskip

\medskip

\end{abstract}

\section{Motivation and introduction}

\setcounter{equation}{0}

In this paper we denote a cardinality of a set $A$ by $\left\vert
A\right\vert $. The cardinality $\aleph _{0}$ we denote by $\infty $. We
will denote a disjoint union of sets by symbol $\uplus $.

We will consider all problems and prove all results in this paper in the
general situation of many-sorted universal algebras. Also the results of
this paper are more important in the studying of varieties of many-sorted
universal algebras.

We will define the many-sorted universal algebras as in \cite%
{TsurkovManySorted}. We suppose that there is a set $\Gamma $ of names of
sorts. A many-sorted algebra is a set $H$ which has "sorting", that is, the
mapping $\eta _{H}:H\rightarrow \Gamma $. The set $\eta _{H}^{-1}\left(
i\right) $ is the set of elements of the sort $i$ of the algebra $H$, where $%
i\in \Gamma $. We denote $\eta _{H}^{-1}\left( i\right) =H^{\left( i\right)
} $. An element from $H^{(i)}$ we denote by $h^{(i)}$, in order to emphasize
that it is an element of the sort $i$. The situation, when exists $i\in
\Gamma $ such that $H^{(i)}=\varnothing $, is possible. For every subset $%
Y\subset H$ we denote the set $\eta _{H}\left( Y\right) \subset \Gamma $ by $%
\Gamma _{Y}$. In particular we denote $\mathrm{im}\eta _{H}=\left\{ i\in
\Gamma \mid H^{\left( i\right) }\neq \varnothing \right\} $ by $\Gamma _{H}$.

We denote by $\Omega $ a signature (set of operations) of our algebras. In
many-sorted case every operation $\omega \in \Omega $ has the type $\tau
_{\omega }=\left( i_{1},\ldots ,i_{n};j\right) $, where $n\in 
\mathbb{N}
$, $i_{1},\ldots ,i_{n},j\in \Gamma $. Operation $\omega \in \Omega $ of the
type $\left( i_{1},\ldots ,i_{n};j\right) $ is a mapping $\omega :H^{\left(
i_{1}\right) }\times \ldots \times H^{\left( i_{n}\right) }\rightarrow
H^{\left( j\right) }$.

\begin{definition}
We suppose that many-sorted algebras $A$ and $B$ have a same set $\Gamma $
of names of sorts and a same signature $\Omega $. We say that the mapping $%
\varphi :A\rightarrow B$ is a \textbf{homomorphism} from algebra $A$ to
algebra $B$ if the equation%
\begin{equation}
\eta _{A}=\eta _{B}\varphi  \label{hom_1}
\end{equation}%
fulfills and for every $\omega \in \Omega $, the equation%
\begin{equation}
\varphi \left( \omega \left( a_{1},\ldots ,a_{n}\right) \right) =\omega
\left( \varphi \left( a_{1}\right) ,\ldots ,\varphi \left( a_{n}\right)
\right)  \label{hom_2}
\end{equation}%
holds.
\end{definition}

\noindent We suppose in (\ref{hom_2}) that $\tau _{\omega }=\left(
i_{1},\ldots ,i_{n};j\right) $, $a_{k}\in A^{\left( i_{k}\right) }$, $1\leq
k\leq n$.

The equation (\ref{hom_1}) means that the mapping $\varphi $ transform all
elements of every sort $i\in \Gamma $ of algebra $A$ to elements of same
sort of algebra $B$. Also we can conclude from (\ref{hom_1}) that no exist
homomorphisms from $A$ to $B$ if $\Gamma _{A}\nsubseteq \Gamma _{B}$. We
denote by $\mathrm{Hom}\left( A,B\right) $ the set of homomorphisms from
algebra $A$ to algebra $B$. As we can see, the situation when $\mathrm{Hom}%
\left( A,B\right) =\varnothing $ is possible. We denote $\varphi _{\mid
A^{\left( i\right) }}$ by $\varphi ^{\left( i\right) }$, $i\in \Gamma $. By (%
\ref{hom_1}) we have that $\varphi ^{\left( i\right) }$ is a mapping from $%
A^{\left( i\right) }$ to $B^{\left( i\right) }$.

A congruence $U$ in an universal algebra $H$ with a set $\Gamma $ of names
of sorts and a signature $\Omega $ is a subset $U\subseteq
\biguplus\limits_{i\in \Gamma }\left( H^{\left( i\right) }\times H^{\left(
i\right) }\right) $, such that for every $\omega \in \Omega $ with $\tau
_{\omega }=\left( i_{1},\ldots ,i_{n};j\right) $ and for every $\left(
h_{k},g_{k}\right) \in U\cap \left( H^{\left( i_{k}\right) }\times H^{\left(
i_{k}\right) }\right) $, where $1\leq k\leq n$, holds $\left( \omega \left(
h_{1},\ldots ,h_{n}\right) ,\omega \left( g_{1},\ldots ,g_{n}\right) \right)
\in U$. Every congruence $U$ in an algebra $H$ gives as a natural
epimorphism $\delta _{U}:H\rightarrow H/U$ from the algebra $H$ to the
quotient algebra $H/U$. In our consideration we will use without special
reminder the

\begin{lemma}
If $H,G$ two algebras with same set of names of sorts and same signature, $U$
is a congruence in algebra $H$, $V$ is a congruence in algebra $G$, $\varphi
:H\rightarrow G$ is a homomorphism and $\varphi \left( U\right) =\left\{
\left( \varphi \left( h^{\prime }\right) ,\varphi \left( h^{\prime \prime
}\right) \right) \mid \left( h^{\prime },h^{\prime \prime }\right) \in
U\right\} \subseteq V$ then there exists a homomorphism $\psi
:H/U\rightarrow G/V$, such that $\delta _{V}\varphi =\psi \delta _{U}$,
\end{lemma}

\noindent which can be proved very easy.

The definition of a kernel of a homomorphism of many-sorted algebras should
be formulated as follows. A kernel of a homomorphism $\varphi :A\rightarrow
B $ of algebras with a set $\Gamma $ of names of sorts is a subset $\ker
\varphi \subseteq \biguplus\limits_{i\in \Gamma }\left( A^{\left( i\right)
}\times A^{\left( i\right) }\right) $ such that for every $\left(
a,a^{\prime }\right) \in \ker \varphi $ the equality $\varphi \left(
a\right) =\varphi \left( a^{\prime }\right) $ holds. A kernel of every
homomorphism is a congruence.

For a little bit different approach to the concept of a many-sorted algebra
see \cite{Pbook}.

We fix the set of names of sorts $\Gamma $ and the signature $\Omega $. We
take for every $i\in \Gamma $ the countable set $X_{0}^{(i)}=\left\{
x_{1}^{(i)},\ldots ,x_{k}^{(i)},\ldots \right\} $. The elements of sets $%
X_{0}^{(i)}$, $i\in \Gamma $, we can call symbols, or letters, or variables,
or generators of the sort $i$. After this we consider the disjoint union $%
X_{0}=\biguplus\limits_{i\in \Gamma }X_{0}^{(i)}$. We choose the subset $%
X=\biguplus\limits_{i\in \Gamma }X^{(i)}\subset X_{0}$, where $%
X^{(i)}=X_{0}^{(i)}\cap X$. It is possible that there exists $i\in \Gamma $
such that $X^{(i)}=\varnothing $. There exists a mapping $\eta
_{X}:X\rightarrow \Gamma $, such that $\eta _{X}\left( x\right) =i$ if $x\in
X^{(i)}\subset X$. Similar to above we can denote $\mathrm{im}\eta
_{X}=\Gamma _{X}$.

We consider the algebra of terms with a set of names of sorts $\Gamma $ and
a signature $\Omega $ over the set (alphabet) $X$. We denote this algebra by 
$\mathfrak{T}\left( \Gamma ,\Omega ,X\right) $ or by $\mathfrak{T}\left(
X\right) $, if it cannot cause errors. We can define by induction by
construction in algebra $\mathfrak{T}\left( X\right) $ sets of terms of the
sort $i$, where $i\in \Gamma $, i.e., sets $\left( \mathfrak{T}\left(
X\right) \right) ^{(i)}$. Of course, the decomposition $\mathfrak{T}\left(
X\right) =\biguplus\limits_{i\in \Gamma }\left( \mathfrak{T}\left( X\right)
\right) ^{(i)}$ fulfills. Hence, exists a mapping $\eta _{\mathfrak{T}\left(
X\right) }:\mathfrak{T}\left( X\right) \rightarrow \Gamma $, such that $\eta
_{\mathfrak{T}\left( X\right) }\left( t\right) =i$ if $t\in \left( \mathfrak{%
T}\left( X\right) \right) ^{(i)}$. The equality $\eta _{\mathfrak{T}\left(
X\right) \mid X}=\eta _{X}$ holds.

It is easy to prove by induction by construction the following

\begin{proposition}
\label{free_empty}For every $X\subset X_{0}$, every algebra $H$ with
signature $\Omega $ such that $\Gamma _{X}\subseteq \Gamma _{H}$ and every 
\textit{mapping }$\varphi ^{\ast }:X\rightarrow H$\textit{\ such that }$\eta
_{X}=\eta _{H}\varphi ^{\ast }$\textit{, there exists unique homomorphism }$%
\varphi :\mathfrak{T}\left( X\right) \rightarrow H$\textit{\ such that }$%
\varphi _{\mid X}=\varphi ^{\ast }$\textit{.}
\end{proposition}

From this proposition we can easy prove the projective property of the
algebra of terms

\begin{proposition}
\label{term_proj}Let $\mathfrak{T}\left( \Gamma ,\Omega ,X\right) =\mathfrak{%
T}\left( X\right) $ the algebra of terms with a set of names of sorts $%
\Gamma $ and a signature $\Omega $ over the set $X$ and $A,B$ two algebras
with same set of names of sorts and same signature. If there exist a
homomorphism $\varphi :\mathfrak{T}\left( X\right) \rightarrow B$ and an
epimorphism $\alpha :A\rightarrow B$ then there exists a homomorphism $\chi :%
\mathfrak{T}\left( X\right) \rightarrow A$, such that $\alpha \chi =\varphi $%
.
\end{proposition}

We suppose that $X=\left\{ x_{1},\ldots ,x_{r}\right\} $ and $f_{1}\left(
x_{1},\ldots ,x_{r}\right) =f_{1}$, $f_{2}\left( x_{1},\ldots ,x_{r}\right)
=f_{2}\in \left( \mathfrak{T}\left( X\right) \right) ^{(i)}$ for certain $%
i\in \Gamma _{\mathfrak{T}\left( X\right) }$. We say that \textit{identity} $%
f_{1}=f_{2}$ holds in algebra $H$ with the set of names of sorts $\Gamma $
and the signature $\Omega $ or algebra $H$ satisfies the identity $%
f_{1}=f_{2}$ if for every $h_{1},\ldots ,h_{r}\in H$ such that $\eta
_{X}\left( x_{j}\right) =\eta _{H}\left( h_{j}\right) $, $1\leq j\leq r$,
the equality $f_{1}\left( h_{1},\ldots ,h_{r}\right) =f_{2}\left(
h_{1},\ldots ,h_{r}\right) $ holds, or, in other words, if for every $%
\varphi \in \mathrm{Hom}\left( \mathfrak{T}\left( X\right) ,H\right) $ the
equality $\varphi \left( f_{1}\right) =\varphi \left( f_{2}\right) $ holds.
We denote this fact by $H\models \left( f_{1}=f_{2}\right) $. In particular $%
H\models \left( f_{1}=f_{2}\right) $ if $\mathrm{Hom}\left( \mathfrak{T}%
\left( X\right) ,H\right) =\varnothing $. Sometimes we will denote
identities briefly: without brackets.

We consider a set of identities $I$. This set can by infinite. But for every 
$\left( f_{1}=f_{2}\right) \in I$ exists finite $X\subset X_{0}$ such that $%
f_{1},f_{2}\in \mathfrak{T}\left( X\right) $. We say that an algebra $H$
satisfies the set of identities $I$ if for every $\left( f_{1}=f_{2}\right)
\in I$ the $H\models \left( f_{1}=f_{2}\right) $ fulfills. This fact we
denote by $H\models I$. The class of all algebras $H$ with set of names of
sorts $\Gamma $ and signature $\Omega $ such that $H\models I$ we call a 
\textit{variety defined by the set of identities} $I$ and denote by $%
Var\left( I\right) $. Therefore, if we consider some variety of algebras, we
suppose that set of names of sorts $\Gamma $ and signature $\Omega $ are
known.

On the other hand we can consider a class of algebras $\Theta $ and the set
of all identities $\mathfrak{I}\left( \Theta \right) $ which hold in all
algebras of $\Theta $. The subset of all identities of $\mathfrak{I}\left(
\Theta \right) $ which contain only variables from a subset $X\subset X_{0}$
we denote by $\mathfrak{I}\left( \Theta ,X\right) $. If we consider the
identity $\left( f_{1}=f_{2}\right) $ as a pair $\left( f_{1},f_{2}\right) $
then for every $X\subset X_{0}$ we have that $\mathfrak{I}\left( \Theta
,X\right) $ is a congruence in $\mathfrak{T}\left( X\right) $.

The proof of the following facts is an easy exercise in logic rather than in
algebra.

\begin{claim}
\label{cl}If $\Lambda ,\Lambda _{1},\Lambda _{2}$ some classes of algebras
with set of names of sorts $\Gamma $ and signature $\Omega $ and $%
I,I_{1},I_{2}$ some sets of identities, i.e.,\linebreak $I,I_{1},I_{2}%
\subseteq \biguplus\limits_{i\in \Gamma }\left( \mathfrak{T}\left( \Gamma
,\Omega ,X\right) ^{(i)}\times \mathfrak{T}\left( \Gamma ,\Omega ,X\right)
^{(i)}\right) $, then
\end{claim}

\begin{enumerate}
\item \label{c_1}$I_{1}\subseteq I_{2}\Longrightarrow Var\left( I_{1}\right)
\supseteq Var\left( I_{2}\right) $\textit{,}

\item \label{c_2}$\Lambda _{1}\subseteq \Lambda _{2}\Longrightarrow 
\mathfrak{I}\left( \Lambda _{1}\right) \supseteq \mathfrak{I}\left( \Lambda
_{2}\right) $\textit{,}

\item $Var\left( \Lambda \right) =Var\left( \mathfrak{I}\left( \Lambda
\right) \right) \supseteq \Lambda $\textit{,}

\item $\mathfrak{I}\left( Var\left( I\right) \right) \supseteq I$\textit{,}

\item $\mathfrak{I}\left( Var\left( \mathfrak{I}\left( \Lambda \right)
\right) \right) =\mathfrak{I}\left( \Lambda \right) $\textit{,}

\item \label{c_6}$Var\left( \mathfrak{I}\left( Var\left( I\right) \right)
\right) =Var\left( I\right) $\textit{, in particular, if }$\Theta $\textit{\
is a variety of algebras, then }$Var\left( \mathfrak{I}\left( \Theta \right)
\right) =\Theta $\textit{,}

\item \label{c_7}\textit{if }$\left( f_{1}=f_{2}\right) \in \mathfrak{I}%
\left( \Lambda ,X\right) $ and $\varphi \in \mathrm{Hom}\left( \mathfrak{T}%
\left( X\right) ,\mathfrak{T}\left( Y\right) \right) $, then $\left( \varphi
\left( f_{1}\right) =\varphi \left( f_{2}\right) \right) \in \mathfrak{I}%
\left( \Lambda ,Y\right) $,

\item \label{c_8}\textit{if }$\Theta _{1},\Theta _{2}$\textit{\ are
varieties of algebras with same set of names of sorts and same signature
then }$\Theta _{1}\cap \Theta _{2}$\textit{\ is also a variety and for every 
}$X\subset X_{0}$\textit{\ the equality }$\mathfrak{I}\left( \Theta _{1}\cap
\Theta _{2},X\right) =\left\langle \mathfrak{I}\left( \Theta _{1},X\right) ,%
\mathfrak{I}\left( \Theta _{2},X\right) \right\rangle $\textit{\ holds, where%
}\linebreak $\left\langle \mathfrak{I}\left( \Theta _{1},X\right) ,\mathfrak{%
I}\left( \Theta _{2},X\right) \right\rangle $\textit{\ is a congruence
generated by congruences }$\mathfrak{I}\left( \Theta _{1},X\right) $\textit{%
\ and }$\mathfrak{I}\left( \Theta _{2},X\right) $\textit{, i.e., the minimal
congruence which contains }$\mathfrak{I}\left( \Theta _{1},X\right) $\textit{%
\ and }$\mathfrak{I}\left( \Theta _{2},X\right) $\textit{.}
\end{enumerate}

\begin{definition}
We say that algebra $F_{\Theta }\left( Y\right) $, is a \textbf{free algebra}
of some variety $\Theta $, generated by the set of free generators $Y$ if
\end{definition}

\begin{enumerate}
\item $Y\subseteq F_{\Theta }\left( Y\right) $\textit{,}

\item $F_{\Theta }\left( Y\right) \in \Theta $\textit{,}

\item \textit{for every }$H\in \Theta $\textit{\ such that }$\Gamma
_{H}\supseteq \Gamma _{Y}$\textit{\ and for every mapping }$\varphi ^{\ast
}:Y\rightarrow H$\textit{\ such that }$\eta _{Y}=\eta _{H}\varphi ^{\ast }$%
\textit{, there exists unique homomorphism }$\varphi :F_{\Theta }\left(
Y\right) \rightarrow H$\textit{\ such that }$\varphi _{\mid Y}=\varphi
^{\ast }$\textit{.}
\end{enumerate}

By Proposition \ref{free_empty} we have that algebra $\mathfrak{T}\left(
X\right) $ is a free algebra of the variety, which defined by empty set of
identities. This free algebra is generated by the set of free generators $X$.

The natural epimorphism $\mathfrak{T}\left( X\right) \rightarrow \mathfrak{T}%
\left( X\right) /\mathfrak{I}\left( \Theta ,X\right) $ we denote by $\delta
_{\Theta ,X}$. From Proposition \ref{free_empty} and from definition of $%
\mathfrak{I}\left( \Theta ,X\right) $ we can conclude that%
\begin{equation}
\mathfrak{T}\left( X\right) /\mathfrak{I}\left( \Theta ,X\right) \cong
F_{\Theta }\left( Y\right) ,  \label{free_alg_var}
\end{equation}%
where $Y=\delta _{\Theta ,X}\left( X\right) $.

\begin{definition}
We say that a variety $\Theta $ is $i$\textbf{-degenerate}, where $i\in
\Gamma $, if $\left( x_{1}^{\left( i\right) }=x_{2}^{\left( i\right)
}\right) \in \mathfrak{I}\left( \Theta \right) $. If for every $i\in \Gamma $
the inclusion $\left( x_{1}^{\left( i\right) }=x_{2}^{\left( i\right)
}\right) \in \mathfrak{I}\left( \Theta \right) $ holds, then we call the
variety $\Theta $ \textbf{degenerate}. We say that a variety $\Theta $ is $i$%
\textbf{-nondegenerate}, where $i\in \Gamma $, if $\left( x_{1}^{\left(
i\right) }=x_{2}^{\left( i\right) }\right) \notin \mathfrak{I}\left( \Theta
\right) $. A variety $\Theta $ is called \textbf{nondegenerate} if it is $i$%
-nondegenerate for every $i\in \Gamma $.
\end{definition}

\begin{proposition}
\label{les2}If a variety $\Theta $ is $i$-degenerate, then for every $H\in
\Theta $ the inclusion $\left\vert H^{\left( i\right) }\right\vert \in
\left\{ 0,1\right\} $ holds.
\end{proposition}

\begin{proof}
We suppose that $\left\vert H^{\left( i\right) }\right\vert >1$. It means
that there are $h_{1}^{\left( i\right) },h_{2}^{\left( i\right) }\in
\left\vert H^{\left( i\right) }\right\vert $. We consider the algebra $%
\mathfrak{T}\left( x_{1}^{\left( i\right) },x_{2}^{\left( i\right) }\right) $
and the mapping $\varphi ^{\ast }:\left\{ x_{1}^{\left( i\right)
},x_{2}^{\left( i\right) }\right\} \rightarrow H$, such that $\varphi ^{\ast
}\left( x_{j}^{\left( i\right) }\right) =h_{j}^{\left( i\right) }$, $j=1,2$.
By Proposition \ref{free_empty} there exists a homomorphism $\varphi :%
\mathfrak{T}\left( x_{1}^{\left( i\right) },x_{2}^{\left( i\right) }\right)
\rightarrow H$, such that $\varphi _{\mid \left\{ x_{1}^{\left( i\right)
},x_{2}^{\left( i\right) }\right\} }=\varphi ^{\ast }$. $\left(
x_{1}^{\left( i\right) }=x_{2}^{\left( i\right) }\right) \in \mathfrak{I}%
\left( \Theta \right) $, so $\varphi \left( x_{1}^{\left( i\right) }\right)
=\varphi \left( x_{2}^{\left( i\right) }\right) $. It means that $%
h_{1}^{\left( i\right) }=h_{2}^{\left( i\right) }$ and gives a contradiction
with $\left\vert H^{\left( i\right) }\right\vert >1$.
\end{proof}

\begin{proposition}
\label{X_injection}Let $X\subset X_{0}$. A variety $\Theta $ is $i$%
-nondegenerate, if and only if the natural epimorphism $\delta _{\Theta ,X}:%
\mathfrak{T}\left( X\right) \rightarrow \mathfrak{T}\left( X\right) /%
\mathfrak{I}\left( \Theta ,X\right) =F_{\Theta }\left( \delta _{\Theta
,X}\left( X\right) \right) $ is an injection on $X^{(i)}$.
\end{proposition}

\begin{proof}
We suppose that the variety $\Theta $ is $i$-nondegenerate, i.e., $\left(
x_{1}^{\left( i\right) }=x_{2}^{\left( i\right) }\right) \notin \mathfrak{I}%
\left( \Theta \right) $, but there are $x_{j_{1}}^{\left( i\right)
},x_{j_{2}}^{\left( i\right) }\in X^{(i)}$ such that $x_{j_{1}}^{\left(
i\right) }\neq x_{j_{2}}^{\left( i\right) }$ and $\delta _{\Theta ,X}\left(
x_{j_{1}}^{\left( i\right) }\right) =\delta _{\Theta ,X}\left(
x_{j_{2}}^{\left( i\right) }\right) $. It means that $\left(
x_{j_{1}}^{\left( i\right) }=x_{j_{2}}^{\left( i\right) }\right) \in 
\mathfrak{I}\left( \Theta ,X\right) $. There exists an endomorphism $\alpha $
of $\mathfrak{T}\left( X\right) $ such that $\alpha \left( x_{j_{1}}^{\left(
i\right) }\right) =x_{1}^{\left( i\right) }$ and $\alpha \left(
x_{j_{2}}^{\left( i\right) }\right) =x_{2}^{\left( i\right) }$. Hence $%
\left( x_{1}^{\left( i\right) }=x_{2}^{\left( i\right) }\right) \in 
\mathfrak{I}\left( \Theta \right) $.

We suppose that the variety $\Theta $ is $i$-degenerate. $F_{\Theta }\left(
\delta _{\Theta ,X}\left( X\right) \right) \in \Theta $, so, by Proposition %
\ref{les2}, $\left\vert \left( F_{\Theta }\left( \delta _{\Theta ,X}\left(
X\right) \right) \right) ^{\left( i\right) }\right\vert <2$. Hence, for
every $x_{j_{1}}^{\left( i\right) },x_{j_{2}}^{\left( i\right) }\in X^{(i)}$
the equality $\delta _{\Theta ,X}\left( x_{j_{1}}^{\left( i\right) }\right)
=\delta _{\Theta ,X}\left( x_{j_{2}}^{\left( i\right) }\right) $ holds.
\end{proof}

We can prove from definition of a free algebra of variety, that if $%
F_{\Theta }\left( Y\right) ,F_{\Theta }\left( Z\right) $ are free algebras
of the variety $\Theta $ with the sets of free generators $%
Y=\biguplus\limits_{i\in \Gamma }Y^{\left( i\right) }$, $Z=\biguplus%
\limits_{i\in \Gamma }Z^{\left( i\right) }$ correspondingly, and for every $%
i\in \Gamma $ the equality $\left\vert Y^{\left( i\right) }\right\vert
=\left\vert Z^{\left( i\right) }\right\vert $ holds, then $F_{\Theta }\left(
Y\right) \cong F_{\Theta }\left( Z\right) $.

\begin{proposition}
\label{up_to_isomorphism}Up to isomorphism all free algebras of a variety $%
\Theta $ have a form (\ref{free_alg_var}).
\end{proposition}

\begin{proof}
Let $F_{\Theta }\left( Z\right) $ is a free algebra of the variety $\Theta $%
, where $Z=\biguplus\limits_{i\in \Gamma }Z^{\left( i\right) }$. We take the
set $X=\biguplus\limits_{i\in \Gamma }X^{\left( i\right) }$, such that for
every $i\in \Gamma $ the equality $\left\vert X^{\left( i\right)
}\right\vert =\left\vert Z^{\left( i\right) }\right\vert $ holds. We
consider the algebra of terms $\mathfrak{T}\left( X\right) $ and the free
algebra $F_{\Theta }\left( Y\right) $ of the variety $\Theta $ where $%
Y=\delta _{\Theta ,X}\left( X\right) $. We have that $Y=\delta _{\Theta
,X}\left( X\right) =\biguplus\limits_{i\in \Gamma }\delta _{\Theta ,X}\left(
X^{\left( i\right) }\right) $. We denote $\delta _{\Theta ,X}\left(
X^{\left( i\right) }\right) $ by $Y^{\left( i\right) }$. If for any $i\in
\Gamma $ we have that $\Theta $ is $i$-nondegenerate variety, then by
Proposition \ref{X_injection} we have that $\left\vert X^{\left( i\right)
}\right\vert =\left\vert Y^{\left( i\right) }\right\vert =\left\vert
Z^{\left( i\right) }\right\vert $. If for $i\in \Gamma $ we have that $%
\Theta $ is $i$-degenerate variety then by Proposition \ref{les2} $%
\left\vert Z^{\left( i\right) }\right\vert =0$ or $\left\vert Z^{\left(
i\right) }\right\vert =1$. We have in the first case that $\left\vert
X^{\left( i\right) }\right\vert =0$ and $\left\vert Y^{\left( i\right)
}\right\vert =0$. We have in the second case that $\left\vert X^{\left(
i\right) }\right\vert =1$ and $\left\vert Y^{\left( i\right) }\right\vert =1$%
. Therefore in all theses cases we have that $\left\vert Y^{\left( i\right)
}\right\vert =\left\vert Z^{\left( i\right) }\right\vert $. Hence $F_{\Theta
}\left( Z\right) \cong F_{\Theta }\left( Y\right) $.
\end{proof}

Now we consider two varieties $\Theta $ and $\Delta $ such that $\Theta
\supseteq \Delta $. Similar to Claim \ref{cl}, item \ref{c_2} we obtain the
inclusion $\mathfrak{I}\left( \Theta ,X\right) \subseteq \mathfrak{I}\left(
\Delta ,X\right) $. We have by the Third Theorem of Isomorphism that 
\begin{equation*}
\left( \mathfrak{T}\left( X\right) /\mathfrak{I}\left( \Theta ,X\right)
\right) /\left( \mathfrak{I}\left( \Delta ,X\right) /\mathfrak{I}\left(
\Theta ,X\right) \right) =F_{\Theta }\left( \delta _{\Theta ,X}\left(
X\right) \right) /\left( \mathfrak{I}\left( \Delta ,X\right) /\mathfrak{I}%
\left( \Theta ,X\right) \right) \cong
\end{equation*}%
\begin{equation*}
\mathfrak{T}\left( X\right) /\mathfrak{I}\left( \Delta ,X\right) \cong
F_{\Delta }\left( \delta _{\Delta ,X}\left( X\right) \right) .
\end{equation*}

For every $X=\biguplus\limits_{i\in \Gamma }X^{(i)}\subset X_{0}$ we
consider all identities of the form $\delta _{\Theta ,X}\left( f_{1}\right)
=\delta _{\Theta ,X}\left( f_{2}\right) $, where $f_{1},f_{2}\in \mathfrak{T}%
\left( X\right) $, which hold in all algebras of the variety $\Delta $. As
above the set of all these identities can be considered as congruence in $%
F_{\Theta }\left( \delta _{\Theta ,X}\left( X\right) \right) $ and we denote
it by $\mathfrak{I}_{\Theta }\left( \Delta ,X\right) $. For every algebra $%
H\in \Theta $ and every homomorphism $\varphi \in \mathrm{Hom}\left( 
\mathfrak{T}\left( X\right) ,H\right) $ there exists unique homomorphism $%
\psi \in \mathrm{Hom}\left( F_{\Theta }\left( \delta _{\Theta ,X}\left(
X\right) \right) ,H\right) $ such that $\varphi =\psi \delta _{\Theta ,X}$.
From this fact we immediately conclude that $\left( \mathfrak{I}\left(
\Delta ,X\right) /\mathfrak{I}\left( \Theta ,X\right) \right) =\mathfrak{I}%
_{\Theta }\left( \Delta ,X\right) $ and%
\begin{equation}
F_{\Delta }\left( \delta _{\Delta ,X}\left( X\right) \right) \cong F_{\Theta
}\left( \delta _{\Theta ,X}\left( X\right) \right) /\mathfrak{I}_{\Theta
}\left( \Delta ,X\right) .  \label{sub_free}
\end{equation}

Now we will prove one proposition which is a generalization of item \ref{c_7}
from Claim \ref{cl}.

\begin{proposition}
\label{rel_hom}\textit{If }$\Theta $ some variety of algebras, $\Delta
\subseteq \Theta $ its subvariety, $\left( g_{1}=g_{2}\right) \in \mathfrak{I%
}_{\Theta }\left( \Delta ,X\right) $ and $\varphi \in \mathrm{Hom}\left(
F_{\Theta }\left( \delta _{\Theta ,X}\left( X\right) \right) ,F_{\Theta
}\left( \delta _{\Theta ,Y}\left( Y\right) \right) \right) $, then $\left(
\varphi \left( g_{1}\right) =\varphi \left( g_{2}\right) \right) \in 
\mathfrak{I}_{\Theta }\left( \Delta ,Y\right) $.
\end{proposition}

\begin{proof}
If $\left( g_{1}=g_{2}\right) \in \mathfrak{I}_{\Theta }\left( \Delta
,X\right) $ then $g_{i}=\delta _{\Theta ,X}\left( f_{i}\right) $, where $%
f_{i}\in \mathfrak{T}\left( X\right) $, $i=1,2$, and for every $H\in \Delta $
the $H\models \left( f_{1}=f_{2}\right) $ holds. By Proposition \ref%
{term_proj}, for every $\varphi \in \mathrm{Mor}_{\Theta ^{0}}\left(
F_{\Theta }\left( \delta _{\Theta ,X}\left( X\right) \right) ,F_{\Theta
}\left( \delta _{\Theta ,Y}\left( Y\right) \right) \right) $ there exists a
homomorphism $\widetilde{\varphi }:\mathfrak{T}\left( X\right) \rightarrow 
\mathfrak{T}\left( Y\right) $, such that $\varphi \delta _{\Theta ,X}=\delta
_{\Theta ,Y}\widetilde{\varphi }$. By Claim \ref{cl} item \ref{c_7}, the $%
H\models \left( \widetilde{\varphi }\left( f_{1}\right) =\widetilde{\varphi }%
\left( f_{2}\right) \right) $ holds for every $H\in \Delta $. It means that $%
\left( \delta _{\Theta ,Y}\widetilde{\varphi }\left( f_{1}\right) =\delta
_{\Theta ,Y}\widetilde{\varphi }\left( f_{2}\right) \right) \in \mathfrak{I}%
_{\Theta }\left( \Delta ,Y\right) $. But $\delta _{\Theta ,Y}\widetilde{%
\varphi }\left( f_{i}\right) =\varphi \delta _{\Theta ,X}\left( f_{i}\right)
=\varphi \left( g_{i}\right) $, where $i=1,2$. So $\left( \varphi \left(
g_{1}\right) =\varphi \left( g_{2}\right) \right) \in \mathfrak{I}_{\Theta
}\left( \Delta ,Y\right) $.
\end{proof}

We will denote $\mathfrak{I}_{\Theta }\left( \Delta ,X\right) $ by $%
\mathfrak{I}\left( \Delta ,X\right) $, if it cannot cause errors.

This research is motivated by universal algebraic geometry. All definitions
of the basic notions of the universal algebraic geometry can be found, for
example, in \cite{PlotkinVarCat}, \cite{PlotkinNotions}, \cite{PlotkinSame}
and \cite{PP}. Also, there are fundamental papers \cite{BMR}, \cite{MR} and 
\cite{DMR2}, \cite{DMR5}.

The relation between geometric and automorphic equivalences of universal
algebras of some variety $\Theta $ is a one important question of universal
algebraic geometry (see \cite{PlotkinSame} and \cite{PP}). We need for the
studying of this relation, at first, consider a category $\Theta ^{0}$ of
the finitely generated free algebras of a variety $\Theta $. Objects of this
category are the finitely generated free algebras of the variety $\Theta $
and morphisms of this category are their homomorphisms. After this we
consider the quotient group $\mathfrak{A/Y}$, where $\mathfrak{A}$\ is the
group of all automorphisms of the category $\Theta ^{0}$ and $\mathfrak{Y}$
is the subgroup of all inner automorphisms of this category. If this
quotient group is trivial, then the geometric and automorphic equivalences
are coincides in the variety $\Theta $ (see \cite{PlotkinSame}). If the
group $\mathfrak{A/Y}$ is not trivial, then often, but not always, we can
give example of two algebras of the variety $\Theta $ which are automorphic
equivalent but not geometric equivalent. From this place onwards we consider
only finitely generated free algebras of some variety $\Theta $, i.e.,
algebras $F_{\Theta }\left( X\right) $, where $X\subset X_{0}$ and $%
\left\vert X\right\vert <\infty $.

For the computing of the group $\mathfrak{A/Y}$ was elaborated on the simple
but very strong method of the verbal operations (see \cite{PlotkinZhitom}
and \cite{TsurkovManySorted}). This method can be applied only when the

\begin{condition}
\label{basic_cond}For every $\Phi \in \mathfrak{A}$, every $i\in \Gamma $
and every $x^{(i)}\in X_{0}^{(i)}$ the isomorphism $\Phi \left( F_{\Theta
}\left( x^{(i)}\right) \right) \cong F_{\Theta }\left( x^{(i)}\right) $
holds.
\end{condition}

\noindent fulfills in the variety $\Theta $.

The IBN (invariant basis number) property or invariant dimension property
was defined initially in the theory of rings and modules, see, for example, 
\cite[Definition 2.8]{Hungerford}. But we can generalize this property for
an arbitrary variety of universal algebras.

\begin{definition}
\label{IBN-variety}We say that a variety $\Theta $ is an $i$\textbf{%
-IBN-variety}, where $i\in \Gamma $, if from $F_{\Theta }\left( Y\right)
\cong F_{\Theta }\left( Z\right) $, where $Y=\biguplus\limits_{j\in \Gamma
}Y^{\left( j\right) }$, $Z=\biguplus\limits_{j\in \Gamma }Z^{\left( j\right)
}$ are sets of free generators of corresponding algebras , we can conclude
that $\left\vert Y^{\left( i\right) }\right\vert =\left\vert Z^{\left(
i\right) }\right\vert $. A variety $\Theta $ is called \textbf{IBN-variety}
or variety which has an \textbf{IBN propriety} if it is an $i$-IBN-variety
for every $i\in \Gamma $.
\end{definition}

In the case of one-sorted algebras the Condition \ref{basic_cond} is weaker
than the IBN propriety. From the IBN propriety of a variety $\Theta $ we can
conclude by the method of \cite[Section 5]{ShestTsur} that in this variety
the Condition \ref{basic_cond} holds. But exist varieties in which the
Condition \ref{basic_cond} holds though these varieties have no the IBN
propriety. The next example is a folklore of the theory of rings and modules.

\begin{example}
We consider some field $k$ and a vector space $V$ over this field such that $%
\dim _{k}V=\infty $. We denote by $A$ the ring of all linear operators over
the vector space $V$: $A=\mathrm{End}_{k}V$. It is known that variety $_{A}%
\mathfrak{M}$ of all lefts modules over the ring $A$ has not the IBN
propriety.
\end{example}

In this variety hold these isomorphisms: $A\cong A\oplus A\cong A\oplus
A\oplus A\cong \ldots \cong \underset{n\text{ }times}{\underbrace{A\oplus
\ldots \oplus A}}\cong \ldots $. But (E. Aladova) in this variety the
Condition \ref{basic_cond} holds, because in this variety all finitely
generated free algebras are isomorphic.

In the case of many-sorted algebras we can not conclude Condition \ref%
{basic_cond} directly from the IBN propriety.

\begin{example}
$\mathcal{SET-COUP}$ is the variety of the couples of sets.
\end{example}

This is a variety of two-sorted algebras, i.e., $\Gamma =\left\{ 1,2\right\} 
$: elements of the first set of a couple are elements of the first sort,
elements of the second set of a couple are elements of the second sort.
Algebras of this variety has the empty signature and this variety is defined
by the empty set of identities. We denote this variety by $\mathcal{SET-COUP}
$. Section \ref{bf} we will prove that this variety has an IBN propriety.\
We consider the functor $\Phi :\mathcal{SET-COUP}^{0}\rightarrow \mathcal{%
SET-COUP}^{0}$ which "switch" the sets in every couple, or, more formal, the
functor $\Phi $ such that%
\begin{equation*}
\Phi \left( F_{\mathcal{SET-COUP}}\left( X\right) \right) =F_{\mathcal{%
SET-COUP}}\left( Y\right) ,
\end{equation*}%
where $X=\left\{ x_{\alpha _{1}}^{(1)},\ldots x_{\alpha _{m}}^{(1)}\right\}
\uplus \left\{ x_{\beta _{1}}^{(2)},\ldots x_{\beta _{n}}^{(2)}\right\} $, $%
Y=\left\{ x_{\beta _{1}}^{(1)},\ldots x_{\beta _{n}}^{(1)}\right\} \uplus
\left\{ x_{\alpha _{1}}^{(2)},\ldots x_{\alpha _{m}}^{(2)}\right\} $, and%
\begin{equation*}
\Phi \left( f:F_{\mathcal{SET-COUP}}\left( A\right) \rightarrow F_{\mathcal{%
SET-COUP}}\left( C\right) \right) =
\end{equation*}%
\begin{equation*}
g:F_{\mathcal{SET-COUP}}\left( \Phi \left( A\right) \right) \rightarrow F_{%
\mathcal{SET-COUP}}\left( \Phi \left( C\right) \right) ,
\end{equation*}%
where $A=\left\{ x_{\alpha _{1}}^{(1)},\ldots x_{\alpha _{k}}^{(1)}\right\}
\uplus \left\{ x_{\beta _{1}}^{(2)},\ldots x_{\beta _{l}}^{(2)}\right\} $, $%
C=\left\{ x_{\gamma _{1}}^{(1)},\ldots x_{\gamma _{p}}^{(1)}\right\} \uplus
\left\{ x_{\delta _{1}}^{(2)},\ldots x_{\delta _{r}}^{(2)}\right\} $, $\Phi
\left( A\right) =\left\{ x_{\beta _{1}}^{(1)},\ldots x_{\beta
_{l}}^{(1)}\right\} \uplus \left\{ x_{\alpha _{1}}^{(2)},\ldots x_{\alpha
_{k}}^{(2)}\right\} $, $\Phi \left( C\right) =\left\{ x_{\delta
_{1}}^{(1)},\ldots x_{\delta _{r}}^{(1)}\right\} \uplus \left\{ x_{\gamma
_{1}}^{(2)},\ldots x_{\gamma _{p}}^{(2)}\right\} $ and if $f\left( x_{\alpha
_{i}}^{(1)}\right) =x_{\gamma _{j}}^{(1)}$, $f\left( x_{\beta
_{u}}^{(2)}\right) =x_{\delta _{w}}^{(2)}$, then $g\left( x_{\beta
_{u}}^{(1)}\right) =x_{\delta _{w}}^{(1)}$, $g\left( x_{\alpha
_{i}}^{(2)}\right) =x_{\gamma _{j}}^{(2)}$. This functor is the inverse of
itself, but%
\begin{equation*}
\Phi \left( F_{\mathcal{SET-COUP}}\left( x_{1}^{\left( 1\right) }\right)
\right) =F_{\mathcal{SET-COUP}}\left( x_{1}^{\left( 2\right) }\right) \ncong
F_{\mathcal{SET-COUP}}\left( x_{1}^{\left( 1\right) }\right) .
\end{equation*}%
It means that Condition \ref{basic_cond} holds not in this variety.

But in many cases, we can use some additional considerations to deduce
Condition \ref{basic_cond} from the IBN propriety of some variety. See, for
example, \cite[Section 5]{TsurkovManySorted} about a variety of all actions
of semigroups over sets, variety of all automatons, variety of all
representations of groups, and \cite[Section 4]{TsurkovSubvarLie} about a
very wide class of subvarieties of the variety of all representation of Lie
algebras.

Therefore, the proving of the IBN propriety of some variety is a milestone
in the study of the relation between geometric and automorphic equivalences
of algebras of this variety.

\section{Simple cases\label{bf}}

\setcounter{equation}{0}

In this section we consider such varieties for which the IBN property can be
proved directly.

\begin{example}
\label{vector_space}The variety of all vector spaces over a fixed field $k$.
\end{example}

We consider vector spaces over a fixed field $k$ as one-sorted algebras. For
every scalar $\lambda \in k$ the multiplication of vectors by this scalar we
consider as one unary operation. A linear map from one vector space to other
can be uniquely defined by images of basic vectors. Therefore all vector
spaces over a fixed field $k$ are free algebras of this variety and the
vectors of basis are free generators of these algebras. It is known that
isomorphic vector spaces have same cardinality of bases, therefore this
variety has the IBN property.

We will prove this

\begin{theorem}
\label{terms}If $\mathfrak{T}\left( X\right) \cong \mathfrak{T}\left(
Y\right) $ then for every $i\in \Gamma $ the equality $\left\vert X^{\left(
i\right) }\right\vert =\left\vert Y^{\left( i\right) }\right\vert $ holds.
\end{theorem}

\begin{proof}
We will define by induction by construction the length of terms, i.e.,
elements of algebra $\mathfrak{T}\left( X\right) $. We will denote a length
of $t\in \mathfrak{T}\left( X\right) $ by $l\left( t\right) $. We define
that length $0$ have only generators, i.e., elements of the set $X$, and
constants of the signature $\Omega $. We suppose that we just defined what
terms have a length $j$ for every $j<n$, where $j,n\in 
\mathbb{N}
$. Now we we will defined what terms have a length $n$. We consider an
operation $\omega \in \Omega $ such that the arity $r$ of this operation is
positive, i.e., $r>0$. If $t_{1},\ldots ,t_{r}\in \mathfrak{T}\left(
X\right) $, such that $l\left( t_{1}\right) ,\ldots ,l\left( t_{r}\right) <n$
and $\max \left\{ l\left( t_{1}\right) ,\ldots ,l\left( t_{r}\right)
\right\} =n-1$, then $l\left( \omega \left( t_{1},\ldots ,t_{r}\right)
\right) =n$.

Now we consider a situation when there exists a homomorphism $\varphi :%
\mathfrak{T}\left( X\right) \rightarrow \mathfrak{T}\left( Y\right) $. The
length of terms is defined both in $\mathfrak{T}\left( X\right) $ and $%
\mathfrak{T}\left( Y\right) $. We will prove by induction by length of term
that for every $t\in \mathfrak{T}\left( X\right) $ the inequality $l\left(
t\right) \leq l\left( \varphi \left( t\right) \right) $ holds. If $l\left(
t\right) =0$, then this inequality is trivial. We suppose that we proved
this inequality when $l\left( t\right) <n$, where $n\in 
\mathbb{N}
$. Let $l\left( t\right) =n$. It means by our definition that $t=\omega
\left( t_{1},\ldots ,t_{r}\right) $, where $\omega \in \Omega $, the arity
of the operation $\omega $ is equal to $r$, $t_{1},\ldots ,t_{r}\in 
\mathfrak{T}\left( X\right) $ and $\max \left\{ l\left( t_{1}\right) ,\ldots
,l\left( t_{r}\right) \right\} =n-1$. So, $\varphi \left( t\right) =\omega
\left( \varphi \left( t_{1}\right) ,\ldots ,\varphi \left( t_{r}\right)
\right) $. By induction hypothesis we have that $l\left( t_{i}\right) \leq
l\left( \varphi \left( t_{i}\right) \right) $, where $1\leq i\leq r$.
Therefore $\max \left\{ l\left( \varphi \left( t_{1}\right) \right) ,\ldots
,l\left( \varphi \left( t_{r}\right) \right) \right\} \geq n-1$ and $l\left(
\varphi \left( t\right) \right) =\max \left\{ l\left( \varphi \left(
t_{1}\right) \right) ,\ldots ,l\left( \varphi \left( t_{r}\right) \right)
\right\} +1\geq n=l\left( t\right) =n$.

Now we suppose that there exists an isomorphism $\varphi :\mathfrak{T}\left(
X\right) \rightarrow \mathfrak{T}\left( Y\right) $. We will prove that for
every $i\in \Gamma $ there exists a bijection $\psi _{i}:X^{(i)}\rightarrow
Y^{(i)}$. We consider $x^{(i)}\in X^{(i)}$ and will prove that $\varphi
\left( x^{(i)}\right) \in Y^{(i)}$. If $\varphi \left( x^{(i)}\right) =t\in 
\mathfrak{T}\left( Y\right) $, then $\varphi ^{-1}\left( t\right) =x^{(i)}$. 
$\varphi ^{-1}:\mathfrak{T}\left( Y\right) \rightarrow \mathfrak{T}\left(
X\right) $ is also a homomorphism. It was proved that $l\left( t\right) \leq
l\left( \varphi ^{-1}\left( t\right) \right) =l\left( x^{(i)}\right) =0$.
Therefore $t\in Y$ or $t=c$ is a constant. We have in the second case that $%
t=c$ is also an element of $\mathfrak{T}\left( X\right) $ and $\varphi
\left( x^{(i)}\right) =t=c=\varphi \left( c\right) $, which contradicts the
injectivity of the mapping $\varphi $. In the first case we have by (\ref%
{hom_1}) that $\varphi \left( x^{(i)}\right) \in Y^{(i)}$. It means that
there exists a mapping $\psi _{i}=\varphi _{\mid X^{(i)}}:X^{(i)}\rightarrow
Y^{(i)}$. We have by symmetry that also there exists a mapping $\chi
_{i}=\left( \varphi ^{-1}\right) _{\mid Y^{(i)}}:Y^{(i)}\rightarrow X^{(i)}$%
. It is clear that $\chi _{i}\psi _{i}=id_{X^{(i)}}$ and $\psi _{i}\chi
_{i}=id_{Y^{(i)}}$. Therefore $\psi _{i}$ is a bijection. This completes the
proof.
\end{proof}

We conclude from this theorem that

\begin{example}
Varieties defined by empty set of identities
\end{example}

are IBN-varieties.

In particular

\begin{example}
\label{sets}The variety of all sets
\end{example}

is an IBN-variety.

\begin{example}
The variety of all graphs.
\end{example}

We consider graphs as $2$-sorted algebras. The first sort is the sort of
edges of a graph, the second sort is the sort of vertices of a graph. The
signature of graphs contain $2$ unary operations: $h$ and $t$. The operation 
$h$ define a head of an edge and the operation $t$ define a tail of an edge.
These operations have a same type: $\tau _{h}=\tau _{t}=\left( 1;2\right) $.
The variety of all graphs defined by empty set of identities, so this
variety is an IBN-variety.

\begin{example}
The variety of all automatons.
\end{example}

This variety was considered in \cite{TsurkovManySorted}. We consider
automatons as $2$-sorted algebras.$\ $The first sort is a sort of input
signals, the second sort is a sort of statements of an automaton, the third
sort is a sort of output signals. $\Omega =\left\{ \ast ,\circ \right\} $.
The operation $\ast $ gives as a new statement of an automaton according to
an input signal and a previous statement of automaton: $\tau _{\ast }=\left(
1,2;2\right) $. The operation $\circ $ gives an output signal according to
an input signal and a statement of automaton. $\tau _{\circ }=\left(
1,2;3\right) $. The variety of all automatons defined by empty set of
identities, so this variety is an IBN-variety.

\section{Functor $\mathcal{D}$. General results.}

\setcounter{equation}{0}

In this section we study IBM property of a variety by properties of its
subvarieties. This method seems to us rather strong. We can choose such
subvarieties that an observation of their free algebras is very simple. The
first result in this direction was obtained by \cite{Fujiwara}. Were
considered the one-sorted algebras and was proved the

\begin{theorem}
\label{Fuji}The variety $\Theta $ is an IBM-variety if there exists a
subvariety $\Delta \subseteq \Theta $ such that $\Delta $ is a nondegenerate
and for every $F\in \mathrm{Ob}\Delta ^{0}$ the inequality $\left\vert
F\right\vert <\infty $ holds.
\end{theorem}

We will consider one example as an application of this theorem. We will use
this example below.

\begin{example}
\label{gr}A nondegenerate variety of groups.
\end{example}

We will prove that every nondegenerate variety of groups is an IBN-variety.
We denote some nondegenerate variety of groups, which we will consider, by $%
\Theta $. By \cite[Theorem 15.1.10]{KM}, $\Theta $ defined by set of
identities%
\begin{equation*}
W=\left\{ x_{1}^{d}=1,u_{1}=1,u_{2}=1,\ldots \right\} ,
\end{equation*}%
where $d\in 
\mathbb{N}
$, $d\neq 1$, $u_{i}$ are elements of commutators of free groups $F\left(
X\right) $, such that $\left\vert X\right\vert <\infty $. We will consider
two cases:

\begin{case}
$2\mid d$, in particular $d=0$.
\end{case}

In this case we consider the subvariety $\Delta \subseteq \Theta $ defined
by identity%
\begin{equation*}
x_{1}^{2}=1.
\end{equation*}%
This is a nondegenerate variety, because this variety contains the group $%
\mathbb{Z}
_{2}$. By theorem about finitely generated abelian group every group of this
variety generated by $n$ generators contains $2^{n}$ elements. Therefore $%
\Delta $ fulfills conditions of Theorem \ref{Fuji} and $\Theta $ is an
IBN-variety.

\begin{case}
$2\nmid d$.
\end{case}

In this case there exists a prime number $p>2$ such that $p\mid d$. We
consider the subvariety $\Delta \subseteq \Theta $ defined by identities%
\begin{equation*}
x_{1}^{p}=1,x_{1}x_{2}=x_{2}x_{1}.
\end{equation*}%
As above, this is a nondegenerate variety, because this variety contains the
group $%
\mathbb{Z}
_{p}$, and every group of this variety generated by $n$ generators contains $%
p^{n}$ elements. And as above, we conclude from Theorem \ref{Fuji} that $%
\Theta $ is an IBN-variety.

In this section we generalize the Theorem \ref{Fuji}.

We consider only free algebras of varieties which have form (\ref%
{free_alg_var}), because by Proposition \ref{up_to_isomorphism} all free
algebras of varieties are isomorphic to the free algebras of this form.

We consider a variety of universal algebras $\Theta $ with a set $\Gamma $
of names of sorts. If $\Delta \subseteq \Theta $ then by (\ref{sub_free}) $%
F_{\Delta }\left( \delta _{\Delta ,X}\left( X\right) \right) \cong F_{\Theta
}\left( \delta _{\Theta ,X}\left( X\right) \right) /\mathfrak{I}_{\Theta
}\left( \Delta ,X\right) $. We will denote by $\delta _{\Delta ,X}^{\Theta
}:F_{\Theta }\left( \delta _{\Theta ,X}\left( X\right) \right) \rightarrow
F_{\Theta }\left( \delta _{\Theta ,X}\left( X\right) \right) /\mathfrak{I}%
_{\Theta }\left( \Delta ,X\right) $ the natural epimorphism.

\begin{lemma}
\label{un_hom}If $\Delta \subseteq \Theta $ then for every $\varphi \in 
\mathrm{Mor}_{\Theta ^{0}}\left( F_{\Theta }\left( \delta _{\Theta ,X}\left(
X\right) \right) ,F_{\Theta }\left( \delta _{\Theta ,Y}\left( Y\right)
\right) \right) $ there exists unique $\varphi ^{\ast }\in \mathrm{Mor}%
_{\Delta ^{0}}\left( F_{\Delta }\left( \delta _{\Delta ,X}\left( X\right)
\right) ,F_{\Delta }\left( \delta _{\Delta ,Y}\left( Y\right) \right)
\right) $ such that $\delta _{\Delta ,Y}^{\Theta }\varphi =\varphi ^{\ast
}\delta _{\Delta ,X}^{\Theta }$.
\end{lemma}

\begin{proof}
By Proposition \ref{rel_hom} $\varphi \left( \mathfrak{I}_{\Theta }\left(
\Delta ,X\right) \right) \subseteq \mathfrak{I}_{\Theta }\left( \Delta
,Y\right) $. So, there exists a homomorphism%
\begin{equation*}
\varphi ^{\ast }:F_{\Theta }\left( \delta _{\Theta ,X}\left( X\right)
\right) /\mathfrak{I}_{\Theta }\left( \Delta ,X\right) \cong F_{\Delta
}\left( \delta _{\Delta ,X}\left( X\right) \right) \rightarrow 
\end{equation*}%
\begin{equation*}
F_{\Theta }\left( \left( \delta _{\Theta ,Y}\left( Y\right) \right) \right) /%
\mathfrak{I}_{\Theta }\left( \Delta ,Y\right) \cong F_{\Delta }\left( \delta
_{\Delta ,Y}\left( Y\right) \right) 
\end{equation*}%
such that $\delta _{\Delta ,Y}^{\Theta }\varphi =\varphi ^{\ast }\delta
_{\Delta ,X}^{\Theta }$.

If there are $\varphi _{1},\varphi _{2}\in \mathrm{Mor}_{\Delta ^{0}}\left(
F_{\Delta }\left( \delta _{\Delta ,X}\left( X\right) \right) ,F_{\Delta
}\left( \delta _{\Theta ,Y}\left( Y\right) \right) \right) $ such that $%
\delta _{\Delta ,Y}^{\Theta }\varphi =\varphi _{1}\delta _{\Delta
,X}^{\Theta }$ and $\delta _{\Delta ,Y}^{\Theta }\varphi =\varphi _{2}\delta
_{\Delta ,X}^{\Theta }$, then $\varphi _{1}\delta _{\Delta ,X}^{\Theta
}=\varphi _{2}\delta _{\Delta ,X}^{\Theta }$, and, because $\delta _{\Delta
,X}^{\Theta }$ is an epimorphism, $\varphi _{1}=\varphi _{2}$.
\end{proof}

Now we will define the functor $\mathcal{D}_{\Delta }^{\Theta }:$ $\Theta
^{0}\rightarrow \Delta ^{0}$. We define the mapping $\mathcal{D}_{\Delta
}^{\Theta }$ from $\mathrm{Ob}\Theta ^{0}$ to $\mathrm{Ob}\Delta ^{0}$ and
from $\mathrm{Mor}_{\Theta ^{0}}$ to $\mathrm{Mor}_{\Delta ^{0}}$ as follows:

\begin{enumerate}
\item $\mathcal{D}_{\Delta }^{\Theta }\left( F_{\Theta }\left( \delta
_{\Theta ,X}\left( X\right) \right) \right) =F_{\Delta }\left( \delta
_{\Delta ,X}\left( X\right) \right) \cong F_{\Theta }\left( \delta _{\Theta
,X}\left( X\right) \right) /\mathfrak{I}_{\Theta }\left( \Delta ,X\right) $,

\item if $\varphi \in \mathrm{Mor}_{\Theta ^{0}}\left( F_{\Theta }\left(
\delta _{\Theta ,X}\left( X\right) \right) ,F_{\Theta }\left( \delta
_{\Theta ,Y}\left( Y\right) \right) \right) $ then%
\begin{equation*}
\mathcal{D}_{\Delta }^{\Theta }\left( \varphi \right) =\varphi ^{\ast }\in 
\mathrm{Mor}_{\Delta ^{0}}\left( F_{\Delta }\left( \delta _{\Delta ,X}\left(
X\right) \right) ,F_{\Delta }\left( \delta _{\Delta ,Y}\left( Y\right)
\right) \right)
\end{equation*}%
such that $\delta _{\Delta ,Y}^{\Theta }\varphi =\varphi ^{\ast }\delta
_{\Delta ,X}^{\Theta }$.
\end{enumerate}

By Lemma \ref{un_hom} the $\mathcal{D}_{\Delta }^{\Theta }\left( \varphi
\right) $ exists and is well-defined.

\begin{corollary}
\label{functor}The mapping $\mathcal{D}_{\Delta }^{\Theta }$ is a functor.
\end{corollary}

\begin{proof}
For every $F_{\Theta }\left( \delta _{\Theta ,X}\left( X\right) \right) \in 
\mathrm{Ob}\Theta ^{0}$ we have that $\delta _{\Delta ,X}^{\Theta
}id_{F_{\Theta }\left( \delta _{\Theta ,X}\left( X\right) \right)
}=id_{F_{\Delta }\left( \delta _{\Delta ,X}\left( X\right) \right) }\delta
_{\Delta ,X}^{\Theta }$, so $\mathcal{D}_{\Delta }^{\Theta }\left(
id_{F_{\Theta }\left( \delta _{\Theta ,X}\left( X\right) \right) }\right)
=id_{\mathcal{D}_{\Delta }^{\Theta }\left( F_{\Theta }\left( \delta _{\Theta
,X}\left( X\right) \right) \right) }$.

If $F_{\Theta }\left( \delta _{\Theta ,X_{1}}\left( X_{1}\right) \right)
,F_{\Theta }\left( \delta _{\Theta ,X_{2}}\left( X_{2}\right) \right)
,F_{\Theta }\left( \delta _{\Theta ,X_{3}}\left( X_{3}\right) \right) \in 
\mathrm{Ob}\Theta ^{0}$ and%
\begin{equation*}
\varphi _{1}\in \mathrm{Mor}_{\Theta ^{0}}\left( F_{\Theta }\left( \delta
_{\Theta ,X_{1}}\left( X_{1}\right) \right) ,F_{\Theta }\left( \delta
_{\Theta ,X_{2}}\left( X_{2}\right) \right) \right) ,
\end{equation*}%
\begin{equation*}
\varphi _{2}\in \mathrm{Mor}_{\Theta ^{0}}\left( F_{\Theta }\left( \delta
_{\Theta ,X_{2}}\left( X_{2}\right) \right) ,F_{\Theta }\left( \delta
_{\Theta ,X_{3}}\left( X_{3}\right) \right) \right) ,
\end{equation*}%
then by Lemma \ref{un_hom} we have that $\delta _{\Delta ,X_{2}}^{\Theta
}\varphi _{1}=\mathcal{D}_{\Delta }^{\Theta }\left( \varphi _{1}\right)
\delta _{\Delta ,X_{1}}^{\Theta }$ and $\delta _{\Delta ,X_{3}}^{\Theta
}\varphi _{2}=\mathcal{D}_{\Delta }^{\Theta }\left( \varphi _{2}\right)
\delta _{\Delta ,X_{2}}^{\Theta }$. So%
\begin{equation*}
\delta _{\Delta ,X_{3}}^{\Theta }\varphi _{2}\varphi _{1}=\mathcal{D}%
_{\Delta }^{\Theta }\left( \varphi _{2}\right) \delta _{\Delta
,X_{2}}^{\Theta }\varphi _{1}=\mathcal{D}_{\Delta }^{\Theta }\left( \varphi
_{2}\right) \mathcal{D}_{\Delta }^{\Theta }\left( \varphi _{1}\right) \delta
_{\Delta ,X_{1}}^{\Theta }.
\end{equation*}%
Therefore $\mathcal{D}_{\Delta }^{\Theta }\left( \varphi _{2}\varphi
_{1}\right) =\mathcal{D}_{\Delta }^{\Theta }\left( \varphi _{2}\right) 
\mathcal{D}_{\Delta }^{\Theta }\left( \varphi _{1}\right) $.
\end{proof}

\setcounter{corollary}{0}

\begin{theorem}
\label{main}The variety $\Theta $ is an IBM-variety if for every $i\in
\Gamma $ there exists a subvariety $\Delta _{i}\subseteq \Theta $ such that $%
\Delta _{i}$ is an $i$-nondegenerate $i$-IBN-variety.
\end{theorem}

\begin{proof}
The variety $\Theta $, which fulfills the condition of this theorem, is a
nondegenerate variety by Claim \ref{cl}, item \ref{c_2}.

We consider $F_{\Theta }\left( \delta _{\Theta ,X}\left( X\right) \right)
,F_{\Theta }\left( \delta _{\Theta ,Y}\left( Y\right) \right) \in \mathrm{Ob}%
\Theta ^{0}$ such that $F_{\Theta }\left( \delta _{\Theta ,X}\left( X\right)
\right) \cong F_{\Theta }\left( \delta _{\Theta ,Y}\left( Y\right) \right) $%
. The functor $\mathcal{D}_{\Delta _{i}}^{\Theta }$ transforms this
isomorphism to the isomorphism $F_{\Delta _{i}}\left( \delta _{\Delta
_{i},X}\left( X\right) \right) \cong F_{\Delta _{i}}\left( \delta _{\Delta
_{i},Y}\left( Y\right) \right) $. $\Delta _{i}$ is an $i$-IBN-variety,
so\linebreak $\left\vert \left( \delta _{\Delta _{i},X}\left( X\right)
\right) ^{\left( i\right) }\right\vert =\left\vert \left( \delta _{\Delta
_{i},Y}\left( Y\right) \right) ^{\left( i\right) }\right\vert $. $\delta
_{\Delta _{i},X}$ and $\delta _{\Delta _{i},Y}$ are homomorphisms, hence, we
have by (\ref{hom_1}) that $\left( \delta _{\Delta _{i},X}\left( X\right)
\right) ^{\left( i\right) }=\left( \delta _{\Delta _{i},X}\left( X^{\left(
i\right) }\right) \right) $, $\left( \delta _{\Delta _{i},Y}\left( Y\right)
\right) ^{\left( i\right) }=\left( \delta _{\Delta _{i},Y}\left( Y^{\left(
i\right) }\right) \right) $. $\Delta _{i}$ is an $i$-nondegenerate,
therefore, by Proposition \ref{X_injection}, $\left\vert X^{\left( i\right)
}\right\vert =\left\vert Y^{\left( i\right) }\right\vert $. $\Theta $ is a
nondegenerate variety, so $\left\vert \delta _{\Theta ,X}\left( X^{\left(
i\right) }\right) \right\vert =\left\vert \delta _{\Theta ,Y}\left(
Y^{\left( i\right) }\right) \right\vert $ and, because $\delta _{\Theta ,X}$
and $\delta _{\Theta ,Y}$ are homomorphisms $\left\vert \left( \delta
_{\Theta ,X}\left( X\right) \right) ^{\left( i\right) }\right\vert
=\left\vert \left( \delta _{\Theta ,Y}\left( Y\right) \right) ^{\left(
i\right) }\right\vert $. This equality holds for every $i\in \Gamma $, hence 
$\Theta $ is an IBM-variety.
\end{proof}

Now we will demonstrate the application of this theorem to some varieties of
one-sorted and many-sorted algebras. In the following examples we will
consider wide enough classes of varieties of algebras. For arbitrary variety 
$\Theta $ from theses wide classes and for every name of sort $i\in \Gamma $
of elements of theses algebras we will define by some identity a variety $%
\Delta _{i}$ which will by an $i$-IBN-variety. If for every $i\in \Gamma $
the varieties will be $i$-nondegenerate we can conclude that the variety $%
\Theta $ is an IBN variety.

\begin{example}
\label{lin_alg}A nondegenerate variety of linear algebras over a fixed field 
$k$.
\end{example}

Linear algebras are the vector spaces with an additional binary operation
called multiplication, which is a bilinear. We consider a nondegenerate
variety $\Theta $ of linear algebras over a fixed field $k$. We take a
subvariety $\Delta \subseteq \Theta $ defined in $\Theta $ by identity%
\begin{equation}
x_{1}x_{2}=0,  \label{null_mul}
\end{equation}%
i.e., subvariety of linear algebras with null multiplication. We will prove

\begin{proposition}
\label{nondeg_linalg}The variety $\Delta $ is a nondegenerate variety.
\end{proposition}

\begin{proof}
The variety of all linear algebras we denote by $\Lambda $. We consider some
free algebra $F_{\Theta }\left( \delta _{\Theta ,X}\left( X\right) \right)
=F $ of the variety $\Theta $. This algebra is isomorphic to the quotient
algebra $F\left( Y\right) /I\left( Y\right) $, where $Y=\delta _{\Lambda
,X}\left( X\right) $, $F\left( Y\right) =F_{\Lambda }\left( \delta _{\Lambda
,X}\left( X\right) \right) $ is an absolutely free linear algebra generate
by set of generators $Y$, $I\left( Y\right) =\mathfrak{I}_{\Lambda }\left(
\Theta ,X\right) $ is a multihomogeneous ideal of $F\left( Y\right) $. The
proof of this fact see, for example, in \cite[Ch. VII, 3.1, Proposition 2]%
{BahturinStruct}. We suppose that $I\left( Y\right) $ contains generators of
degree $1$, i.e., $\lambda _{i_{1}}y_{i_{1}}+\ldots +\lambda
_{i_{n}}y_{i_{n}}\in I\left( Y\right) $, where $y_{i_{1}},\ldots
,y_{i_{n}}\in Y$, $\lambda _{i_{1}},\ldots ,\lambda _{i_{n}}\in k\setminus
\left\{ 0\right\} $. There exists an endomorphism $\alpha $ of $F\left(
Y\right) $ such that $\alpha \left( y_{i_{1}}\right) =y_{i_{1}}$, $\alpha
\left( y_{i_{2}}\right) =\ldots =\alpha \left( y_{i_{n}}\right) =0$. By
Proposition \ref{rel_hom} we have that $\alpha \left( \lambda
_{i_{1}}y_{i_{1}}+\ldots +\lambda _{i_{n}}y_{i_{n}}\right) =\lambda
_{i_{1}}y_{i_{1}}\in I\left( Y\right) $, so $y_{i_{1}}\in I\left( Y\right) $
and $\Theta $ is a degenerate variety. From this contradiction we conclude
that $I\left( Y\right) \subseteq \left( F\left( Y\right) \right) ^{2}$. By
Claim \ref{cl} item \ref{c_8} we have that $F_{\Delta }\left( \delta
_{\Theta ,\Delta }\left( X\right) \right) \cong F\left( Y\right) /\left(
F\left( Y\right) \right) ^{2}$, i.e., this is a vector space with the basis $%
Y$ equipped by null multiplication. From Proposition \ref{les2} we conclude
that $\Delta $ is a nondegenerate variety.
\end{proof}

Every algebra of the subvariety $\Delta $ is a vector space equipped by null
multiplication. Every homomorphism of algebras of this subvariety are linear
map. So, as in Example \ref{vector_space}, every algebra of this subvariety
is a free algebras and the vectors of the basis of this vector space are
free generators. Hence, as in Example \ref{vector_space}, the subvariety $%
\Delta $ has the IBN property. Therefore by Theorem \ref{main} the variety $%
\Theta $ has the IBN property.

\begin{example}
\label{semi}(A. Sivatski) A nondegenerate variety of semigroups.
\end{example}

Let $\Theta $ is a nondegenerate variety of semigroups. We consider a
subvariety $\Delta \subseteq \Theta $ defined in $\Theta $ by identity%
\begin{equation}
x_{1}x_{2}=x_{1}.  \label{l_mult}
\end{equation}%
We suppose that the variety $\Delta $ is a nondegenerate variety. Then the
free semigroup $F_{\Delta }\left( X\right) $ of the subvariety $\Delta $ is
a set $X$ with the multiplication defined by identity (\ref{l_mult}) and
every mapping of sets from $F_{\Delta }\left( X\right) =X$ to $F_{\Delta
}\left( Y\right) =Y$ $\ $is a homomorphism of these semigroups. Therefore,
by Example \ref{sets}, if $F_{\Delta }\left( X\right) \cong F_{\Delta
}\left( Y\right) $ then $\left\vert X\right\vert =\left\vert Y\right\vert $
and the variety $\Theta $ has the IBN property.

If $\Theta $ is a nondegenerate variety of commutative\ semigroups then the
subvariety defined in $\Theta $ by identity (\ref{l_mult}) is a degenerate
variety, because from the identities $x_{1}x_{2}=x_{2}x_{1}$ and $%
x_{1}x_{2}=x_{1}$ we conclude the identity $x_{1}=x_{2}$. So we need use
another approach for the proving of IBN property of varieties of
commutative\ semigroups.

\begin{example}
\label{comm_semi}A nondegenerate variety of commutative semigroups.
\end{example}

Let $\Theta $ is a nondegenerate variety of commutative semigroups. We
consider a subvariety $\Delta \subseteq \Theta $ defined in $\Theta $ by
identity%
\begin{equation*}
x^{2}=x.
\end{equation*}%
We suppose that the variety $\Delta $ is a nondegenerate variety. We
consider a finite set $X=\left\{ x_{1},\ldots ,x_{n}\right\} $ and a
semigroup $F_{\Delta }\left( X\right) $ which is free semigroup in the
subvariety $\Delta $ and generated by the set $X$ of free generators. Every
element of $F_{\Delta }\left( X\right) $ can be presented in the form $%
x_{i_{1}}x_{i_{2}}\ldots x_{i_{m}}$, where $1\leq m\leq n$, $%
i_{1}<i_{2}<\ldots <i_{m}$. Hence $\left\vert F_{\Delta }\left( X\right)
\right\vert <\infty $ and by Theorem \ref{Fuji} the variety $\Theta $ is an
IBM-variety.

\section{Examples of varieties of $2$-sorted algebras}

\setcounter{equation}{0}

Now we will consider three examples of varieties of $2$-sorted algebras
i.e., such that $\Gamma =\left\{ 1,2\right\} $. The studying of the IBM
properties of these varieties was started in \cite{Katsov}. Here we present
an approach that is more general and uniform. So the considerations in all
of these examples will be very similar.

A signature $\Omega $ of every algebra which we will study in these examples
fulfills the

\begin{condition}
\label{act_separ}A signature $\Omega $ is separated into three classes of
algebraic operations: $\Omega =\Omega ^{\left( 1\right) }\uplus \Omega
^{\left( 2\right) }\uplus \Omega _{a}$.The operations from the first class $%
\Omega ^{\left( 1\right) }$ give results in the first sort of algebras and
have arguments also only from the first sort. The operations from the second
class $\Omega ^{\left( 2\right) }$ give results in the second sort of
algebras and have arguments also only from the second sort. And the third
class $\Omega _{a}$ contain only one binary operation which we call "action"
and denote by $\circ $, this operation has type $\tau _{\circ }=\left(
1,2;2\right) $.
\end{condition}

Now we will prove some proprieties of algebras which signatures have such
separation.

We consider some signature which fulfills Condition \ref{act_separ}, some
set $X=X^{\left( 1\right) }\uplus X^{\left( 2\right) }\subset X_{0}$ such
that $\left\vert X\right\vert <\infty $ and algebra of terms $\mathfrak{T}%
\left( \left\{ 1,2\right\} ,\Omega ,X\right) =\mathfrak{T}\left( X\right) $.

\begin{proposition}
\label{terms_1_sort}The equality $\mathfrak{T}\left( X\right) ^{\left(
1\right) }=\mathfrak{T}\left( \left\{ 1\right\} ,\Omega ^{\left( 1\right)
},X^{\left( 1\right) }\right) $ holds.
\end{proposition}

\begin{proof}
We will prove the inclusion $\mathfrak{T}\left( X\right) ^{\left( 1\right)
}\subseteq \mathfrak{T}\left( \left\{ 1\right\} ,\Omega ^{\left( 1\right)
},X^{\left( 1\right) }\right) $ by induction by construction. The elements
of the set $X^{\left( 1\right) }$ and all constants of the first sort are
elements of $\mathfrak{T}\left( \left\{ 1\right\} ,\Omega ^{\left( 1\right)
},X^{\left( 1\right) }\right) $. This is the induction basis. Now we will
make the induction step. We consider a term $\omega \left( t_{1},\ldots
,t_{n}\right) \in \mathfrak{T}\left( X\right) ^{\left( 1\right) }$, where $%
t_{1},\ldots ,t_{n}\in \mathfrak{T}\left( X\right) $. So $\tau _{\omega
}=\left( i_{1},\ldots ,i_{n};1\right) $, where $i_{1},\ldots ,i_{n}\in
\left\{ 1,2\right\} $. Our signature fulfills Condition \ref{act_separ},
hence $i_{1}=\ldots =i_{n}=1$. Therefore $\omega \in \Omega ^{\left(
1\right) }$ and $t_{1},\ldots ,t_{n}\in \mathfrak{T}\left( X\right) ^{\left(
1\right) }$. By the induction hypothesis $t_{1},\ldots ,t_{n}\in \mathfrak{T}%
\left( \left\{ 1\right\} ,\Omega ^{\left( 1\right) },X^{\left( 1\right)
}\right) $, so $\omega \left( t_{1},\ldots ,t_{n}\right) \in \mathfrak{T}%
\left( \left\{ 1\right\} ,\Omega ^{\left( 1\right) },X^{\left( 1\right)
}\right) $.

We can prove the inclusion $\mathfrak{T}\left( \left\{ 1\right\} ,\Omega
^{\left( 1\right) },X^{\left( 1\right) }\right) \subseteq \mathfrak{T}\left(
X\right) ^{\left( 1\right) }$ also by induction by construction even simpler
than previous inclusion.
\end{proof}

Now we consider some variety $\Theta $ of algebras which fulfill Condition %
\ref{act_separ}. The inclusion%
\begin{equation*}
\mathfrak{I}\left( \Theta ,X\right) \subseteq \left( \mathfrak{T}\left(
X\right) ^{\left( 1\right) }\times \mathfrak{T}\left( X\right) ^{\left(
1\right) }\right) \uplus \left( \mathfrak{T}\left( X\right) ^{\left(
2\right) }\times \mathfrak{T}\left( X\right) ^{\left( 2\right) }\right) 
\end{equation*}%
holds, because $\mathfrak{I}\left( \Theta ,X\right) $ is a congruence in $%
\mathfrak{T}\left( X\right) $. The intersection $\mathfrak{I}\left( \Theta
,X\right) \cap \left( \mathfrak{T}\left( X\right) ^{\left( 1\right) }\times 
\mathfrak{T}\left( X\right) ^{\left( 1\right) }\right) $ which we denote by $%
\left( I\left( \Theta ,X\right) \right) ^{\left( 1\right) }$ is a congruence
in $\mathfrak{T}\left( X\right) ^{\left( 1\right) }=\mathfrak{T}\left(
\left\{ 1\right\} ,\Omega ^{\left( 1\right) },X^{\left( 1\right) }\right) $,
so elements of $\left( I\left( \Theta ,X\right) \right) ^{\left( 1\right) }$
can be considered as identities in the class of one-sorted algebras with
signature $\Omega ^{\left( 1\right) }$. We denote an union $%
\bigcup\limits_{X\subset X_{0},\left\vert X\right\vert <\infty }\left(
I\left( \Theta ,X\right) \right) ^{\left( 1\right) }$ by $\left( I\left(
\Theta \right) \right) ^{\left( 1\right) }$. If we considered elements of $%
\left( I\left( \Theta \right) \right) ^{\left( 1\right) }$ as identities of
one-sorted algebras with the signature $\Omega ^{\left( 1\right) }$, then
they define the variety of one-sorted algebras with signature $\Omega
^{\left( 1\right) }$, which we denote by $\Theta ^{\left( 1\right) }$.

\begin{proposition}
\label{(1)}If $A=A^{\left( 1\right) }\uplus A^{\left( 2\right) }$ is an
arbitrary algebra of the variety $\Theta $, then $A^{\left( 1\right) }\in
\Theta ^{\left( 1\right) }$.
\end{proposition}

\begin{proof}
By Condition \ref{act_separ} the set $A^{\left( 1\right) }$ is a one-sorted
algebra with signature $\Omega ^{\left( 1\right) }$. We consider some
identity%
\begin{equation}
\left( f_{1}\left( x_{1},\ldots ,x_{r}\right) =f_{2}\left( x_{1},\ldots
,x_{r}\right) \right) \in \left( I\left( \Theta \right) \right) ^{\left(
1\right) }.  \label{r_id}
\end{equation}%
By Proposition \ref{terms_1_sort} all elements of the set $X=\left\{
x_{1},\ldots ,x_{r}\right\} $ are elements of the first sort, so $%
X=X^{\left( 1\right) }=\left\{ x_{1}^{\left( 1\right) },\ldots
,x_{r}^{\left( 1\right) }\right\} $. We consider some homomorphism $\psi :%
\mathfrak{T}\left( \left\{ 1\right\} ,\Omega ^{\left( 1\right) },X^{\left(
1\right) }\right) \rightarrow A^{\left( 1\right) }$. By Proposition \ref%
{free_empty} there exists a homomorphism $\varphi :\mathfrak{T}\left(
X\right) \rightarrow A$, such that $\varphi \left( x_{j}^{\left( 1\right)
}\right) =\psi \left( x_{j}^{\left( 1\right) }\right) $, $1\leq j\leq r$. $%
A\in \Theta $, so $A$ fulfills the identity (\ref{r_id}) as identity of $2$%
-sorted algebras with signature $\Omega $. Therefore these equalities
fulfill:%
\begin{equation*}
\varphi \left( f_{1}\left( x_{1}^{\left( 1\right) },\ldots ,x_{r}^{\left(
1\right) }\right) \right) =\varphi \left( f_{2}\left( x_{1}^{\left( 1\right)
},\ldots ,x_{r}^{\left( 1\right) }\right) \right) ,
\end{equation*}%
\begin{equation*}
f_{1}\left( \varphi \left( x_{1}^{\left( 1\right) }\right) ,\ldots ,\varphi
\left( x_{r}^{\left( 1\right) }\right) \right) =f_{2}\left( \varphi \left(
x_{1}^{\left( 1\right) }\right) ,\ldots ,\varphi \left( x_{r}^{\left(
1\right) }\right) \right) ,
\end{equation*}
\begin{equation*}
f_{1}\left( \psi \left( x_{1}^{\left( 1\right) }\right) ,\ldots ,\psi \left(
x_{r}^{\left( 1\right) }\right) \right) =f_{2}\left( \psi \left(
x_{1}^{\left( 1\right) }\right) ,\ldots ,\psi \left( x_{r}^{\left( 1\right)
}\right) \right) .
\end{equation*}%
$f_{1}\left( x_{1},\ldots ,x_{r}\right) ,f_{2}\left( x_{1},\ldots
,x_{r}\right) \in \mathfrak{T}\left( \left\{ 1\right\} ,\Omega ^{\left(
1\right) },X^{\left( 1\right) }\right) $, hence we obtain an equality%
\begin{equation*}
\psi \left( f_{1}\left( x_{1}^{\left( 1\right) },\ldots ,x_{r}^{\left(
1\right) }\right) \right) =\psi \left( f_{2}\left( x_{1}^{\left( 1\right)
},\ldots ,x_{r}^{\left( 1\right) }\right) \right) .
\end{equation*}%
It means that $A^{\left( 1\right) }$ fulfills the identity (\ref{r_id}) as
identity of one-sorted algebras with signature $\Omega ^{\left( 1\right) }$.
Therefore $A^{\left( 1\right) }\in \Theta ^{\left( 1\right) }$.
\end{proof}

\begin{proposition}
\label{1_nondeg}If $\Theta $ is nondegenerate variety then $\Theta ^{\left(
1\right) }$ is nondegenerate variety.
\end{proposition}

\begin{proof}
We consider the set $X=X^{\left( 1\right) }=\left\{ x_{1}^{\left( 1\right)
},x_{2}^{\left( 1\right) }\right\} $. The variety $\Theta $ is
nondegenerate, so by Proposition \ref{X_injection} the free algebra $%
F_{\Theta }\left( \delta _{\Theta ,X}\left( X\right) \right) $ contains no
less then $2$ elements of the first sort. By Proposition \ref{(1)} $\left(
F_{\Theta }\left( \delta _{\Theta ,X}\left( X\right) \right) \right)
^{\left( 1\right) }\in \Theta ^{\left( 1\right) }$. Hence, by Proposition %
\ref{les2}, the variety $\Theta ^{\left( 1\right) }$ is nondegenerate.
\end{proof}

\begin{proposition}
\label{1_IBN}If $\Theta ^{\left( 1\right) }$ is an IBN-variety, then $\Theta 
$ is an $1$-IBN-variety.
\end{proposition}

\begin{proof}
In the beginning we consider a set $X=X^{\left( 1\right) }\uplus X^{\left(
2\right) }$. There exists a natural epimorphism $\delta _{\Theta ,X}:%
\mathfrak{T}\left( X\right) \rightarrow F_{\Theta }\left( \delta _{\Theta
,X}\left( X\right) \right) $, where $F_{\Theta }\left( \delta _{\Theta
,X}\left( X\right) \right) $ is a free algebra of the variety $\Theta $ with
the set of free generators $\delta _{\Theta ,X}\left( X\right) $. We will
prove that $\left( F_{\Theta }\left( \delta _{\Theta ,X}\left( X\right)
\right) \right) ^{\left( 1\right) }$ is a free algebra of the variety $%
\Theta ^{\left( 1\right) }$ with the set of free generators $\delta _{\Theta
,X}\left( X^{\left( 1\right) }\right) $. By Condition \ref{act_separ} $%
\delta _{\Theta ,X}^{\left( 1\right) }$ is a homomorphism of one-sorted
algebras with signature $\Omega ^{\left( 1\right) }$. By definition $\ker
\delta _{\Theta ,X}=\mathfrak{I}\left( \Theta ,X\right) $, so $\ker \delta
_{\Theta ,X}^{\left( 1\right) }=\mathfrak{I}\left( \Theta ,X\right) \cap
\left( \mathfrak{T}\left( X\right) ^{\left( 1\right) }\times \mathfrak{T}%
\left( X\right) ^{\left( 1\right) }\right) =\left( I\left( \Theta ,X\right)
\right) ^{\left( 1\right) }$. $\delta _{\Theta ,X}$ is an epimorphism, so $%
\mathrm{im}\delta _{\Theta ,X}^{\left( 1\right) }=\left( F_{\Theta }\left(
\delta _{\Theta ,X}\left( X\right) \right) \right) ^{\left( 1\right) }$. By
Proposition \ref{terms_1_sort} the equality $\mathfrak{T}\left( X\right)
^{\left( 1\right) }=\mathfrak{T}\left( \left\{ 1\right\} ,\Omega ^{\left(
1\right) },X^{\left( 1\right) }\right) $ holds. So $\mathrm{im}\delta
_{\Theta ,X}^{\left( 1\right) }=\left( F_{\Theta }\left( \delta _{\Theta
,X}\left( X\right) \right) \right) ^{\left( 1\right) }\cong \mathfrak{T}%
\left( X\right) ^{\left( 1\right) }/\ker \delta _{\Theta ,X}^{\left(
1\right) }=\mathfrak{T}\left( \left\{ 1\right\} ,\Omega ^{\left( 1\right)
},X^{\left( 1\right) }\right) /\left( I\left( \Theta ,X\right) \right)
^{\left( 1\right) }$. So, by (\ref{free_alg_var}),\linebreak $\left(
F_{\Theta }\left( \delta _{\Theta ,X}\left( X\right) \right) \right)
^{\left( 1\right) }$ is a free algebra of the variety $\Theta ^{\left(
1\right) }$ with the set of free generators $\delta _{\Theta ,X}^{\left(
1\right) }\left( X^{\left( 1\right) }\right) =\delta _{\Theta ,X}\left(
X^{\left( 1\right) }\right) $.

Now we suppose that there exists an isomorphism $\varphi :F_{\Theta }\left(
\delta _{\Theta ,X}\left( X\right) \right) \rightarrow F_{\Theta }\left(
\delta _{\Theta ,Y}\left( Y\right) \right) $ of two free algebras of the the
variety $\Theta $. $\varphi ^{\left( 1\right) }:$ $\left( F_{\Theta }\left(
\delta _{\Theta ,X}\left( X\right) \right) \right) ^{\left( 1\right)
}\rightarrow \left( F_{\Theta }\left( \delta _{\Theta ,Y}\left( Y\right)
\right) \right) ^{\left( 1\right) }$ is a bijection and a homomorphism of
one-sorted algebras with signature $\Omega ^{\left( 1\right) }$, so $\varphi
^{\left( 1\right) }$ is an isomorphism. As we proved above $\left( F_{\Theta
}\left( \delta _{\Theta ,X}\left( X\right) \right) \right) ^{\left( 1\right)
}$ and $\left( F_{\Theta }\left( \delta _{\Theta ,Y}\left( Y\right) \right)
\right) ^{\left( 1\right) }$ are free algebras of the variety $\Theta
^{\left( 1\right) }$ with sets of free generators $\delta _{\Theta ,X}\left(
X^{\left( 1\right) }\right) $ and $\delta _{\Theta ,Y}\left( Y^{\left(
1\right) }\right) $ respectively. If $\Theta ^{\left( 1\right) }$ is an
IBN-variety, then $\left\vert \delta _{\Theta ,X}\left( X^{\left( 1\right)
}\right) \right\vert =\left\vert \delta _{\Theta ,Y}\left( Y^{\left(
1\right) }\right) \right\vert $. The equalities $\delta _{\Theta ,X}\left(
X^{\left( 1\right) }\right) =\left( \delta _{\Theta ,X}\left( X\right)
\right) ^{\left( 1\right) }$ and $\delta _{\Theta ,Y}\left( Y^{\left(
1\right) }\right) =\left( \delta _{\Theta ,Y}\left( Y\right) \right)
^{\left( 1\right) }$ hold, so $\Theta $ is an $1$-IBN-variety.
\end{proof}

Now we consider the variety $\Lambda $ defined by identity%
\begin{equation}
x^{\left( 1\right) }\circ x^{\left( 2\right) }=s\left( x^{\left( 2\right)
}\right) \text{,}  \label{tr_act}
\end{equation}%
where $s\left( x^{\left( 2\right) }\right) $ is some term from element $%
x^{\left( 2\right) }$ constructed by operations of $\Omega ^{\left( 2\right)
}$, i.e., element of $\mathfrak{T}\left( \left\{ 2\right\} ,\Omega ^{\left(
2\right) },x^{\left( 2\right) }\right) $. Further, in particular examples,
we will choose the term $s\left( x^{\left( 2\right) }\right) $ in different
ways.

\begin{proposition}
\label{tau_2}If some variety $\Delta $ is a subvariety of $\Lambda $, then
for every $X\subset X_{0}$ and every $f\in \mathfrak{T}\left( X\right)
^{\left( 2\right) }$ there exists $f^{\prime }\in \mathfrak{T}\left( \left\{
2\right\} ,\Omega ^{\left( 2\right) },X^{\left( 2\right) }\right) $ such
that $\delta _{\Delta ,X}^{\left( 2\right) }\left( f\right) =\delta _{\Delta
,X}^{\left( 2\right) }\left( f^{\prime }\right) $, where $\delta _{\Delta
,X}:\mathfrak{T}\left( X\right) \rightarrow F_{\Delta }\left( \delta
_{\Delta ,X}\left( X\right) \right) $ is a natural epimorphism.
\end{proposition}

\begin{proof}
We will prove this fact by induction by construction. For elements of the
set $X^{\left( 2\right) }$ and for constant of the second sort this fact is
trivial. This is the induction basis. Now we will make the induction step.

In the first case we suppose that $f=\omega \left( t_{1},\ldots t_{n}\right) 
$, where $\omega \in \Omega ^{\left( 2\right) }$ and $t_{1},\ldots ,t_{n}\in 
\mathfrak{T}\left( X\right) ^{\left( 2\right) }$. By induction hypothesis
there exist\linebreak $t_{1}^{\prime },\ldots ,t_{n}^{\prime }\in \mathfrak{T%
}\left( \left\{ 2\right\} ,\Omega ^{\left( 2\right) },X^{\left( 2\right)
}\right) $, such that $\delta _{\Delta ,X}^{\left( 2\right) }\left(
t_{i}\right) =\delta _{\Delta ,X}^{\left( 2\right) }\left( t_{i}^{\prime
}\right) $, $1\leq i\leq n$. Therefore%
\begin{equation*}
\delta _{\Delta ,X}^{\left( 2\right) }\left( f\right) =\delta _{\Delta
,X}^{\left( 2\right) }\left( \omega \left( t_{1},\ldots ,t_{n}\right)
\right) =\omega \left( \delta _{\Delta ,X}^{\left( 2\right) }\left(
t_{1}\right) ,\ldots ,\delta _{\Delta ,X}^{\left( 2\right) }\left(
t_{n}\right) \right) =
\end{equation*}%
\begin{equation*}
\omega \left( \delta _{\Delta ,X}^{\left( 2\right) }\left( t_{1}^{\prime
}\right) ,\ldots ,\delta _{\Delta ,X}^{\left( 2\right) }\left( t_{n}^{\prime
}\right) \right) =\delta _{\Delta ,X}^{\left( 2\right) }\left( \omega \left(
t_{1}^{\prime },\ldots ,t_{n}^{\prime }\right) \right) .
\end{equation*}%
$\omega \left( t_{1}^{\prime },\ldots ,t_{n}^{\prime }\right) =f^{\prime
}\in \mathfrak{T}\left( \left\{ 2\right\} ,\Omega ^{\left( 2\right)
},X^{\left( 2\right) }\right) $.

In the second case we suppose that $f=t^{\left( 1\right) }\circ t^{\left(
2\right) }$, where $t^{\left( 1\right) }\in \mathfrak{T}\left( X\right)
^{\left( 1\right) }$, $t^{\left( 2\right) }\in \mathfrak{T}\left( X\right)
^{\left( 2\right) }$. By the induction hypothesis there exists $t^{\prime
}\in \mathfrak{T}\left( \left\{ 2\right\} ,\Omega ^{\left( 2\right)
},X^{\left( 2\right) }\right) $, such that $\delta _{\Delta ,X}^{\left(
2\right) }\left( t^{\left( 2\right) }\right) =\delta _{\Delta ,X}^{\left(
2\right) }\left( t^{\prime }\right) $. $F_{\Delta }\left( \delta _{\Delta
,X}\left( X\right) \right) $ is an algebra of the variety $\Delta $, so, by
item \ref{c_2} of Claim \ref{cl}, it fulfills the identity (\ref{tr_act}).
Therefore $\delta _{\Delta ,X}^{\left( 2\right) }\left( f\right) =\delta
_{\Delta ,X}^{\left( 2\right) }\left( t^{\left( 1\right) }\right) \circ
\delta _{\Delta ,X}^{\left( 2\right) }\left( t^{\left( 2\right) }\right)
=s\left( \delta _{\Delta ,X}^{\left( 2\right) }\left( t^{\left( 2\right)
}\right) \right) =s\left( \delta _{\Delta ,X}^{\left( 2\right) }\left(
t^{\prime }\right) \right) =\delta _{\Delta ,X}^{\left( 2\right) }\left(
s\left( t^{\prime }\right) \right) $. $s\left( t^{\prime }\right) =f^{\prime
}\in \mathfrak{T}\left( \left\{ 2\right\} ,\Omega ^{\left( 2\right)
},X^{\left( 2\right) }\right) $.
\end{proof}

\begin{corollary}
\label{cor_tau_2}Under the conditions of this proposition the equality%
\begin{equation*}
\delta _{\Delta ,X}^{\left( 2\right) }\left( \mathfrak{T}\left( X\right)
^{\left( 2\right) }\right) =\left( F_{\Delta }\left( \delta _{\Delta
,X}\left( X\right) \right) \right) ^{\left( 2\right) }=\delta _{\Delta
,X}^{\left( 2\right) }\left( \mathfrak{T}\left( \left\{ 2\right\} ,\Omega
^{\left( 2\right) },X^{\left( 2\right) }\right) \right) 
\end{equation*}%
holds.
\end{corollary}

\setcounter{corollary}{0}

Now we denote by $\Delta $ a subvariety of the variety $\Theta $, defined by
the identity (\ref{tr_act}). $\mathfrak{I}\left( \Lambda \right) \subseteq 
\mathfrak{I}\left( \Delta \right) $, so $\Lambda =Var\left( \mathfrak{I}%
\left( \Lambda \right) \right) \supseteq Var\left( \mathfrak{I}\left( \Delta
\right) \right) =\Delta $, by items \ref{c_1} and \ref{c_6} of Claim \ref{cl}%
. Therefore $\Delta $ is a subject of Proposition \ref{tau_2}.

For every $X\subset X_{0}$ we denote $\mathfrak{I}\left( \Delta ,X\right)
\cap \left( \mathfrak{T}\left( \left\{ 2\right\} ,\Omega ^{\left( 2\right)
},X^{\left( 2\right) }\right) \times \mathfrak{T}\left( \left\{ 2\right\}
,\Omega ^{\left( 2\right) },X^{\left( 2\right) }\right) \right) $ by $\left(
I\left( \Delta ,X\right) \right) ^{\left( 2\right) }$. An union $%
\bigcup\limits_{X\subset X_{0},\left\vert X\right\vert <\infty }\left(
I\left( \Delta ,X\right) \right) ^{\left( 2\right) }$ we denote by $\left(
I\left( \Delta \right) \right) ^{\left( 2\right) }$. $\left( I\left( \Delta
,X\right) \right) ^{\left( 2\right) }$ is a congruence in $\mathfrak{T}%
\left( \left\{ 2\right\} ,\Omega ^{\left( 2\right) },X^{\left( 2\right)
}\right) $, so elements of $\left( I\left( \Delta \right) \right) ^{\left(
2\right) }$ can be considered as identities of one-sorted algebras with the
signature $\Omega ^{\left( 2\right) }$. The variety defined by these
identities we denote by $\Delta ^{\left( 2\right) }$. We define the variety $%
\Delta ^{\left( 1\right) }$ in the same way as we previously defined the
variety $\Theta ^{\left( 1\right) }$.

Now we consider an one-sorted algebra $H^{\left( 1\right) }$ with signature $%
\Omega ^{\left( 1\right) }$, such that $H^{\left( 1\right) }\in \Delta
^{\left( 1\right) }$, and an one-sorted algebra $H^{\left( 2\right) }$ with
signature $\Omega ^{\left( 2\right) }$, such that $H^{\left( 2\right) }\in
\Delta ^{\left( 2\right) }$. We construct a $2$-sorted algebra $H=H^{\left(
1\right) }\uplus H^{\left( 2\right) }$, such that $H^{\left( 1\right) }$ is
a set of all elements of the first sort of algebra $H$ and $H^{\left(
2\right) }$ is a set of all elements of the second sort of algebra $H$. We
define the action elements of $H^{\left( 1\right) }$ over elements of $%
H^{\left( 2\right) }$ this way%
\begin{equation}
\forall h^{\left( 1\right) }\in H^{\left( 1\right) },h^{\left( 2\right) }\in
H^{\left( 2\right) }\hspace{0.1in}h^{\left( 1\right) }\circ h^{\left(
2\right) }=s\left( h^{\left( 2\right) }\right) ,  \label{triv_act_axiom}
\end{equation}%
where $s$ is a term from (\ref{tr_act}). By Condition \ref{act_separ} the
algebra $H$ is an algebra of signature $\Omega $.

\begin{proposition}
\label{delta}$H$ is an algebra of the variety $\Delta $.
\end{proposition}

\begin{proof}
We have from (\ref{triv_act_axiom}) that $H$ fulfills the identity (\ref%
{tr_act}), so $H\in \Lambda $. We will prove that algebra $H$ fulfills all
identities of variety $\Delta $.

We consider a set $X=X^{\left( 1\right) }\uplus X^{\left( 2\right) }\subset
X_{0}$ and algebra of terms $\mathfrak{T}\left( X\right) $. We suppose that $%
\left( f,g\right) \in \mathfrak{I}\left( \Delta ,X\right) \cap \left( 
\mathfrak{T}\left( X\right) ^{\left( 2\right) }\times \mathfrak{T}\left(
X\right) ^{\left( 2\right) }\right) \subset \mathfrak{I}\left( \Delta
\right) $. It means that $\left( f,g\right) \in \ker \delta _{\Delta ,X}$,
i.e., $\delta _{\Delta ,X}\left( f\right) =\delta _{\Delta ,X}\left(
g\right) $. By Proposition \ref{tau_2} there exist $f^{\prime },g^{\prime
}\in \mathfrak{T}\left( \left\{ 2\right\} ,\Omega ^{\left( 2\right)
},X^{\left( 2\right) }\right) $ such that $\delta _{\Lambda ,X}\left(
f\right) =\delta _{\Lambda ,X}^{\left( 2\right) }\left( f\right) =\delta
_{\Lambda ,X}^{\left( 2\right) }\left( f^{\prime }\right) =\delta _{\Lambda
,X}\left( f^{\prime }\right) $ and $\delta _{\Lambda ,X}\left( g\right)
=\delta _{\Lambda ,X}^{\left( 2\right) }\left( g\right) =\delta _{\Lambda
,X}^{\left( 2\right) }\left( g^{\prime }\right) =\delta _{\Lambda ,X}\left(
g^{\prime }\right) $. Therefore%
\begin{equation}
\left( f,f^{\prime }\right) ,\left( g,g^{\prime }\right) \in \ker \delta
_{\Lambda ,X}=\mathfrak{I}\left( \Lambda ,X\right) .  \label{praim_lambda}
\end{equation}%
Now we consider an arbitrary homomorphism $\varphi :\mathfrak{T}\left(
X\right) \rightarrow H$. $H\in \Lambda $, so%
\begin{equation}
\varphi \left( f\right) =\varphi \left( f^{\prime }\right) ,\varphi \left(
g\right) =\varphi \left( g^{\prime }\right) .  \label{praim_f_g}
\end{equation}

$\Delta \subseteq \Lambda $ , so, similar to item \ref{c_2} of Claim \ref{cl}%
, $\ker \delta _{\Delta ,X}=\mathfrak{I}\left( \Delta ,X\right) \supseteq
\ker \delta _{\Lambda ,X}=\mathfrak{I}\left( \Lambda ,X\right) $. Hence, by (%
\ref{praim_lambda}) $\delta _{\Delta ,X}\left( f\right) =\delta _{\Delta
,X}^{\left( 2\right) }\left( f\right) =\delta _{\Delta ,X}^{\left( 2\right)
}\left( f^{\prime }\right) =\delta _{\Delta ,X}\left( f^{\prime }\right) $\
and $\delta _{\Delta ,X}\left( g\right) =\delta _{\Delta ,X}^{\left(
2\right) }\left( g\right) =\delta _{\Delta ,X}^{\left( 2\right) }\left(
g^{\prime }\right) =\delta _{\Delta ,X}\left( g^{\prime }\right) $. Hence $%
\left( f^{\prime },g^{\prime }\right) \in \ker \delta _{\Delta ,X}=\mathfrak{%
I}\left( \Delta ,X\right) $. So $\left( f^{\prime },g^{\prime }\right) \in
\left( I\left( \Delta ,X\right) \right) ^{\left( 2\right) }\subset \left(
I\left( \Delta \right) \right) ^{\left( 2\right) }$. Here we consider $%
\left( f^{\prime }=g^{\prime }\right) $ as an identity of one-sorted
algebras with the signature $\Omega ^{\left( 2\right) }$. $\varphi _{\mid 
\mathfrak{T}\left( \left\{ 2\right\} ,\Omega ^{\left( 2\right) },X^{\left(
2\right) }\right) }:\mathfrak{T}\left( \left\{ 2\right\} ,\Omega ^{\left(
2\right) },X^{\left( 2\right) }\right) \rightarrow H^{\left( 2\right) }$ is
a homomorphism from the one-sorted algebra of terms with signature $\Omega
^{\left( 2\right) }$ to the one-sorted algebra $H^{\left( 2\right) }$ with
same signature. $H^{\left( 2\right) }\in \Delta ^{\left( 2\right) }$, so $%
\varphi \left( f^{\prime }\right) =\varphi \left( g^{\prime }\right) $. By (%
\ref{praim_f_g}) we conclude that $\varphi \left( f\right) =\varphi \left(
g\right) $, i.e., $H\models \left( f=g\right) $.

Now we consider $\left( f,g\right) \in \mathfrak{I}\left( \Delta ,X\right)
\cap \left( \mathfrak{T}\left( X\right) ^{\left( 1\right) }\times \mathfrak{T%
}\left( X\right) ^{\left( 1\right) }\right) $. By Proposition \ref%
{terms_1_sort} we can consider $\left( f=g\right) $ as an identity of
one-sorted algebras with the signature $\Omega ^{\left( 1\right) }$. We
consider an arbitrary homomorphism $\varphi :\mathfrak{T}\left( X\right)
\rightarrow H$. By Proposition \ref{terms_1_sort} $\varphi ^{\left( 1\right)
}:\mathfrak{T}\left( \left\{ 1\right\} ,\Omega ^{\left( 1\right) },X^{\left(
1\right) }\right) \rightarrow H^{\left( 1\right) }$ is a homomorphism from
the one-sorted algebra of terms with signature $\Omega ^{\left( 1\right) }$
to the one-sorted algebra $H^{\left( 1\right) }$ with same signature, $%
\left( f,g\right) \in \left( I\left( \Delta \right) \right) ^{\left(
1\right) }$, $H^{\left( 1\right) }\in \Delta ^{\left( 1\right) }$, so $%
\varphi \left( f\right) =\varphi ^{\left( 1\right) }\left( f\right) =\varphi
^{\left( 1\right) }\left( g\right) =\varphi \left( g\right) $, i.e., $%
H\models \left( f=g\right) $.
\end{proof}

\begin{proposition}
\label{delta_2}For every set $X=X^{\left( 1\right) }\uplus X^{\left(
2\right) }\subset X_{0}$ the algebra\linebreak $\left( F_{\Delta }\left(
\delta _{\Delta ,X}\left( X\right) \right) \right) ^{\left( 2\right)
}=\delta _{\Delta ,X}^{\left( 2\right) }\left( \mathfrak{T}\left( \left\{
2\right\} ,\Omega ^{\left( 2\right) },X^{\left( 2\right) }\right) \right) $
is a free algebra of the variety $\Delta ^{\left( 2\right) }$ with the set
of generators $\delta _{\Delta ,X}^{\left( 2\right) }\left( X^{\left(
2\right) }\right) $.
\end{proposition}

\begin{proof}
We consider an arbitrary algebra of the variety $\Delta ^{\left( 2\right) }$%
. This algebra we denote by $H^{\left( 2\right) }$. We consider an arbitrary
mapping $\varphi ^{\ast }:\delta _{\Delta ,X}^{\left( 2\right) }\left(
X^{\left( 2\right) }\right) \rightarrow H^{\left( 2\right) }$. $\Delta $ is
a subject of the Proposition \ref{(1)}, so the algebra $\left( F_{\Delta
}\left( \delta _{\Delta ,X}\left( X\right) \right) \right) ^{\left( 1\right)
}$ which we denote by $H^{\left( 1\right) }$ is an algebra of the variety $%
\Delta ^{\left( 1\right) }$. We consider the set $H=H^{\left( 1\right)
}\uplus H^{\left( 2\right) }$ with operations of the signature $\Omega
^{\left( 1\right) }$ defined in $H^{\left( 1\right) }$ and operations of the
signature $\Omega ^{\left( 2\right) }$ defined in $H^{\left( 2\right) }$. We
define the action elements of $H^{\left( 1\right) }$ over elements of $%
H^{\left( 2\right) }$ by formula (\ref{triv_act_axiom}). For every $%
x^{\left( 1\right) }\in X^{\left( 1\right) }$ we define $\varphi ^{\ast
}\left( \delta _{\Delta ,X}^{\left( 1\right) }\left( x^{\left( 1\right)
}\right) \right) =\delta _{\Delta ,X}^{\left( 1\right) }\left( x^{\left(
1\right) }\right) \in \left( F_{\Delta }\left( \delta _{\Delta ,X}\left(
X\right) \right) \right) ^{\left( 1\right) }=H^{\left( 1\right) }$. By
Proposition \ref{delta} we have that $H\in \Delta $. So there exists a
homomorphism $\varphi :\left( F_{\Delta }\left( \delta _{\Delta ,X}\left(
X\right) \right) \right) \rightarrow H$, such that $\varphi _{\mid \delta
_{\Delta ,X}\left( X\right) }=\left( \varphi ^{\ast }\right) _{\mid \delta
_{\Delta ,X}\left( X\right) }$.%
\begin{equation*}
\varphi ^{\left( 2\right) }:\left( F_{\Delta }\left( \delta _{\Delta
,X}\left( X\right) \right) \right) ^{\left( 2\right) }=\delta _{\Delta
,X}^{\left( 2\right) }\left( \mathfrak{T}\left( \left\{ 2\right\} ,\Omega
^{\left( 2\right) },X^{\left( 2\right) }\right) \right) \rightarrow
H^{\left( 2\right) }
\end{equation*}%
is a homomorphism of one-sorted algebras with signature $\Omega ^{\left(
2\right) }$, such that\linebreak $\left( \varphi ^{\left( 2\right) }\right)
_{\mid \delta _{\Delta ,X}^{\left( 2\right) }\left( X^{\left( 2\right)
}\right) }=\varphi ^{\ast }$, i.e., the equality $\left( \varphi ^{\left(
2\right) }\delta _{\Delta ,X}^{\left( 2\right) }\right) \left( x^{\left(
2\right) }\right) =\left( \varphi ^{\ast }\delta _{\Delta ,X}^{\left(
2\right) }\right) \left( x^{\left( 2\right) }\right) $ holds for every $%
x^{\left( 2\right) }\in X^{\left( 2\right) }$.

We suppose that there is an other homomorphism $\chi :$ $\delta _{\Delta
,X}^{\left( 2\right) }\left( \mathfrak{T}\left( \left\{ 2\right\} ,\Omega
^{\left( 2\right) },X^{\left( 2\right) }\right) \right) \rightarrow
H^{\left( 2\right) }$ of one-sorted algebras with signature $\Omega ^{\left(
2\right) }$ such that $\chi _{\mid \delta _{\Delta ,X}^{\left( 2\right)
}\left( X^{\left( 2\right) }\right) }=\varphi ^{\ast }$, i.e., the equality $%
\left( \chi \delta _{\Delta ,X}^{\left( 2\right) }\right) \left( x^{\left(
2\right) }\right) =\left( \varphi ^{\ast }\delta _{\Delta ,X}^{\left(
2\right) }\right) \left( x^{\left( 2\right) }\right) $ holds for every $%
x^{\left( 2\right) }\in X^{\left( 2\right) }$. By Proposition \ref%
{free_empty} we have that $\varphi ^{\left( 2\right) }\delta _{\Delta
,X}^{\left( 2\right) }=\chi \delta _{\Delta ,X}^{\left( 2\right) }$ and,
because $\delta _{\Delta ,X}^{\left( 2\right) }$ is an epimorphism, we
conclude that $\varphi ^{\left( 2\right) }=\chi $.
\end{proof}

\begin{corollary}
\label{delta_2_nondeg}If $\Delta ^{\left( 2\right) }$ is a nondegenerate
variety, then $\Delta $ is $2$-nondegenerate variety.
\end{corollary}

\begin{proof}
$\left( x_{1}^{\left( 2\right) },x_{2}^{\left( 2\right) }\right) \in 
\mathfrak{T}\left( \left\{ 2\right\} ,\Omega ^{\left( 2\right) },\left\{
x_{1}^{\left( 2\right) },x_{2}^{\left( 2\right) }\right\} \right) \times 
\mathfrak{T}\left( \left\{ 2\right\} ,\Omega ^{\left( 2\right) },\left\{
x_{1}^{\left( 2\right) },x_{2}^{\left( 2\right) }\right\} \right) $. So, if $%
\left( x_{1}^{\left( 2\right) },x_{2}^{\left( 2\right) }\right) \in 
\mathfrak{I}\left( \Delta \right) $, then $\left( x_{1}^{\left( 2\right)
},x_{2}^{\left( 2\right) }\right) \in \left( I\left( \Delta \right) \right)
^{\left( 2\right) }$.
\end{proof}

\begin{corollary}
\label{delta_2_IBN}If $\Delta ^{\left( 2\right) }$ is an IBN-variety, then $%
\Delta $ is a $2$-IBN variety.
\end{corollary}

\begin{proof}
We suppose that there exists an isomorphism $\varphi :\left( F_{\Delta
}\left( \delta _{\Delta ,X}\left( X\right) \right) \right) \rightarrow
\left( F_{\Delta }\left( \delta _{\Delta ,Y}\left( Y\right) \right) \right) $%
, where $X=X^{\left( 1\right) }\uplus X^{\left( 2\right) },Y=Y^{\left(
1\right) }\uplus Y^{\left( 2\right) }\subset X_{0}$. Therefore $\varphi
^{\left( 2\right) }:\left( F_{\Delta }\left( \delta _{\Delta ,X}\left(
X\right) \right) \right) ^{\left( 2\right) }\rightarrow \left( F_{\Delta
}\left( \delta _{\Delta ,Y}\left( Y\right) \right) \right) ^{\left( 2\right)
}$ is an isomorphism. $\left( F_{\Delta }\left( \delta _{\Delta ,X}\left(
X\right) \right) \right) ^{\left( 2\right) }$ is a free algebra of the
variety $\Delta ^{\left( 2\right) }$ with the set of generators $\delta
_{\Delta ,X}^{\left( 2\right) }\left( X^{\left( 2\right) }\right) $, $\left(
F_{\Delta }\left( \delta _{\Delta ,Y}\left( Y\right) \right) \right)
^{\left( 2\right) }$ is a free algebra of the variety $\Delta ^{\left(
2\right) }$ with the set of generators $\delta _{\Delta ,Y}^{\left( 2\right)
}\left( Y^{\left( 2\right) }\right) $, $\Delta ^{\left( 2\right) }$ is an
IBN-variety, so $\left\vert \delta _{\Delta ,X}^{\left( 2\right) }\left(
X^{\left( 2\right) }\right) \right\vert =\left\vert \delta _{\Delta
,Y}^{\left( 2\right) }\left( Y^{\left( 2\right) }\right) \right\vert $.
\end{proof}

\setcounter{corollary}{0}

Now we will prove two propositions, which we will use in the consideration
of the following three examples.

\begin{proposition}
\label{nondeg_sets}The variety of all sets has no nondegenerate subvarieties.
\end{proposition}

\begin{proof}
Nontrivial identities in the variety of all sets have form $x_{i}=$ $x_{j}$,
where $i\neq j$. But, if $\Theta $ is a subvariety of the variety of all
sets and $\left( x_{i},x_{j}\right) \in \mathfrak{I}\left( \Theta \right) $,
where $i\neq j$, then $\Theta $ is a degenerate variety.
\end{proof}

\begin{proposition}
\label{nondeg_vectror_spases}The variety of all vector spaces over some
fixed field $k$ has no nondegenerate subvarieties.
\end{proposition}

\begin{proof}
A vector space $F\left( Y\right) $ with basis $Y$ is a free algebra of the
variety of all vector spaces. Nontrivial identities in this variety have
form $\lambda _{i_{1}}y_{i_{1}}+\ldots +\lambda _{i_{n}}y_{i_{n}}=0$, where $%
y_{i_{1}},\ldots ,y_{i_{n}}\in Y$, $\lambda _{i_{1}},\ldots ,\lambda
_{i_{n}}\in k\setminus \left\{ 0\right\} $. We conclude as in the Proof of
the Proposition \ref{nondeg_linalg} that the subvariety of the variety of
all vector spaces defined by this identity is a degenerate variety.
\end{proof}

The following three examples satisfy Condition \ref{act_separ}.

\begin{example}
\label{semigr_actions}Varieties of actions of semigroups over sets.
\end{example}

In this example we consider varieties of actions of semigroups over sets.
Actions of semigroups over sets are $2$-sorted algebras, i.e., $\Gamma
=\left\{ 1,2\right\} $. The first sort is a sort of elements of semigroups,
the second sort is a sort of elements of sets. $\Omega =\left\{ \cdot ,\circ
\right\} $: $\cdot $ is a multiplication in the semigroup, $\circ $ is an
action of the elements of the semigroup over the elements of the set. $\tau
_{\cdot }=\left( 1,1;1\right) $, $\tau _{\circ }=\left( 1,2;2\right) $.

We denote by $\Theta $ some nondegenerate variety of actions of semigroups
over sets.

If $\Theta ^{\left( 1\right) }$ is an IBN-variety, in particular, one from
varieties considered in Examples \ref{semi} and \ref{comm_semi} then, by
Proposition \ref{1_IBN}, $\Theta $ is an $1$-IBN-variety.

As a variety $\Lambda $ we consider the variety of trivial action, i.e., the
variety defined by identity%
\begin{equation*}
x^{\left( 1\right) }\circ x^{\left( 2\right) }=x^{\left( 2\right) }\text{.}
\end{equation*}%
This is identity (\ref{tr_act}) with $s\left( x^{\left( 2\right) }\right)
=x^{\left( 2\right) }$.

If $\Delta ^{\left( 2\right) }$ is a nondegenerate variety, then, by
Proposition \ref{nondeg_sets}, $\Delta ^{\left( 2\right) }$ is a variety of
all sets. Hence, by Example \ref{sets}, $\Delta ^{\left( 2\right) }$ is an
IBN-variety. By Corollaries \ref{delta_2_nondeg} and \ref{delta_2_IBN} from
Proposition \ref{delta_2} we have that $\Delta $ is a $2$-nondegenerate $2$%
-IBN-variety. By Theorem \ref{main} we can conclude in this case that $%
\Theta $ is an IBN-variety.

\begin{example}
\label{Lie_rep}Varieties of \ representations of Lie algebras over a fixed
field $k$.
\end{example}

In this example we consider varieties of representations of Lie algebras
over some fixed field $k$. A representation of Lie algebras we consider as $2
$-sorted algebra. The first sort is a sort of elements of some Lie algebra
over a fixed field $k$, the second sort is a sort of vectors of some vector
space over same field $k$. As in Example \ref{vector_space}, a
multiplication of elements of the Lie algebra and vectors by any scalar $%
\lambda \in k$ we consider as two different unary operations. All operations
over vectors of the vector space have all arguments of the sort $2$ and give
a result of the sort $2$. All operations over elements of the Lie algebra
have all arguments of the sort $1$ and give a result of the sort $1$. An
operation of an action of elements of the Lie algebra over vectors of the
vector space, which we denote by $\circ $, has a type $\tau _{\circ }=\left(
1,2;2\right) $.

We denote by $\Theta $ some nondegenerate variety of representations Lie
algebras over the field $k$. By Proposition \ref{1_nondeg} the variety $%
\Theta ^{\left( 1\right) }$ is nondegenerate. So, $\Theta ^{\left( 1\right)
} $ is a nondegenerate variety of Lie algebras over the field $k$ and, by
Example \ref{lin_alg}, $\Theta ^{\left( 1\right) }$ is an IBN-variety.
Hence, by Proposition \ref{1_IBN}, $\Theta $ is an $1$-IBN-variety.

As a variety $\Lambda $ we consider the variety of null action, i.e., the
variety defined by identity%
\begin{equation*}
x^{\left( 1\right) }\circ x^{\left( 2\right) }=0^{\left( 2\right) }\text{.}
\end{equation*}%
This is identity (\ref{tr_act}) with $s\left( x^{\left( 2\right) }\right)
=0^{\left( 2\right) }$.

\begin{proposition}
\label{Lie_delta_2_nondeg}$\Delta ^{\left( 2\right) }$ is the variety of all
vector spaces over the field $k$.
\end{proposition}

\begin{proof}
In the proof of this proposition we denote elements $\delta _{\Xi ,X}\left(
x\right) $, $\delta _{\Lambda ,X}\left( x\right) $, $\delta _{\Theta
,X}\left( x\right) $, $\delta _{\Delta ,X}\left( x\right) $ and $x\in
X\subset X_{0}$ by same symbol $x$. This cannot cause confusion here.

$\Lambda $ is a nondegenerate variety because for every Lie algebras over
the field $k$ we can define a null action over an arbitrary vector space
over the field $k$.

In \cite[Theorem 3.1.]{TsurkovSubvarLie} (see also \cite{Simonian}) was
proved that for every variety $\Xi $ of representations Lie algebras over
the field $k$ the free algebra $F_{\Xi }\left( X\right) $ of this variety
generated by set $X=X^{\left( 1\right) }\uplus X^{\left( 2\right) }\subset
X_{0}$ has this form:%
\begin{equation*}
\left( F_{\Xi }\left( X\right) \right) ^{\left( 1\right) }=L_{\Xi }\left(
X^{\left( 1\right) }\right) =L\left( X^{\left( 1\right) }\right) /I_{\Xi
}\left( X^{\left( 1\right) }\right) ,
\end{equation*}%
\begin{equation*}
\left( F_{\Xi }\left( X\right) \right) ^{\left( 2\right)
}=\bigoplus\limits_{x^{\left( 2\right) }\in X^{\left( 2\right) }}\left(
A_{\Xi }\left( X^{\left( 1\right) }\right) x^{\left( 2\right) }\right) ,
\end{equation*}%
where $L\left( X^{\left( 1\right) }\right) $ is a free Lie algebra,
generated by the set $X^{\left( 1\right) }$, $I_{\Xi }\left( X^{\left(
1\right) }\right) $ is a multihomogeneous two-sided ideal of this algebra, $%
\bigoplus\limits_{x^{\left( 2\right) }\in X^{\left( 2\right) }}\left( A_{\Xi
}\left( X^{\left( 1\right) }\right) x^{\left( 2\right) }\right) $ is a
direct sum of $\left\vert X^{\left( 2\right) }\right\vert $ copies of $%
A\left( X^{\left( 1\right) }\right) $ cyclic module\linebreak $A_{\Xi
}\left( X^{\left( 1\right) }\right) =A\left( X^{\left( 1\right) }\right)
/B_{\Xi }\left( X^{\left( 1\right) }\right) $, where $A\left( X^{\left(
1\right) }\right) $ is a free associative algebra with unit, generated by
the set $X^{\left( 1\right) }$, $B_{\Xi }\left( X^{\left( 1\right) }\right) $
is a multihomogeneous two-sided ideal of this algebra.

It is clear that $B_{\Lambda }\left( X^{\left( 1\right) }\right) \supseteq
\left\langle X^{\left( 1\right) }\right\rangle $, where $\left\langle
X^{\left( 1\right) }\right\rangle $ is a two-sided ideal of algebra $A\left(
X^{\left( 1\right) }\right) $, generated by all elements of set $X^{\left(
1\right) }$. $\left\langle X^{\left( 1\right) }\right\rangle $ is a maximal
two-sided ideal of $A\left( X^{\left( 1\right) }\right) $, because $A\left(
X^{\left( 1\right) }\right) /\left\langle X^{\left( 1\right) }\right\rangle
\cong k$. $\Lambda $ is a nondegenerate variety, so $B_{\Lambda }\left(
X^{\left( 1\right) }\right) =\left\langle X^{\left( 1\right) }\right\rangle $%
.

Now we consider $B_{\Theta }\left( X^{\left( 1\right) }\right) $. If there
exists $\lambda \in k$ such that $\lambda \neq 0$ and $\lambda \in B_{\Theta
}\left( X^{\left( 1\right) }\right) $, then the $B_{\Theta }\left( X^{\left(
1\right) }\right) =A\left( X^{\left( 1\right) }\right) $ holds. But this
contradicts the fact that $\Theta $ is a nondegenerate variety. Therefore $%
B_{\Theta }\left( X^{\left( 1\right) }\right) \subseteq \left\langle
X^{\left( 1\right) }\right\rangle =B_{\Lambda }\left( X^{\left( 1\right)
}\right) $. By item \ref{c_8} of Claim \ref{cl} we have that $B_{\Delta
}\left( X^{\left( 1\right) }\right) =\left\langle B_{\Theta }\left(
X^{\left( 1\right) }\right) ,B_{\Lambda }\left( X^{\left( 1\right) }\right)
\right\rangle =B_{\Lambda }\left( X^{\left( 1\right) }\right) =\left\langle
X^{\left( 1\right) }\right\rangle $. Hence%
\begin{equation*}
\left( F_{\Delta }\left( X\right) \right) ^{\left( 2\right)
}=\bigoplus\limits_{x^{\left( 2\right) }\in X^{\left( 2\right) }}\left(
\left( A\left( X^{\left( 1\right) }\right) /\left\langle X^{\left( 1\right)
}\right\rangle \right) x^{\left( 2\right) }\right) \cong
\bigoplus\limits_{x^{\left( 2\right) }\in X^{\left( 2\right) }}kx^{\left(
2\right) }
\end{equation*}%
is a vector space over $k$ which has dimension $\left\vert X^{\left(
2\right) }\right\vert $. By Corollary \ref{cor_tau_2} from Proposition \ref%
{tau_2} $\delta _{\Delta ,X}^{\left( 2\right) }\left( \mathfrak{T}\left(
X\right) ^{\left( 2\right) }\right) =\left( F_{\Delta }\left( X\right)
\right) ^{\left( 2\right) }=\delta _{\Delta ,X}^{\left( 2\right) }\left( 
\mathfrak{T}\left( \left\{ 2\right\} ,\Omega ^{\left( 2\right) },X^{\left(
2\right) }\right) \right) $ and by Proposition \ref{delta_2} $\delta
_{\Delta ,X}^{\left( 2\right) }\left( \mathfrak{T}\left( \left\{ 2\right\}
,\Omega ^{\left( 2\right) },X^{\left( 2\right) }\right) \right) $ is a free
algebra of the variety $\Delta ^{\left( 2\right) }$. Now Propositions \ref%
{les2} and \ref{nondeg_vectror_spases} complete the proof.
\end{proof}

\begin{corollary}
$\Delta ^{\left( 2\right) }$ is an IBN-variety.
\end{corollary}

\begin{proof}
See Example \ref{vector_space}.
\end{proof}

\setcounter{corollary}{0}

So $\Delta $ is a $2$-nondegenerate $2$-IBN-variety by Corollaries \ref%
{delta_2_nondeg} and \ref{delta_2_IBN} from Proposition \ref{delta_2}. Now
we can conclude from Theorem \ref{main} that $\Theta $ is an IBN-variety.

\begin{example}
\label{group_repr}Varieties of \ representations of groups over a fixed
field $k$.
\end{example}

In this example we consider varieties of representations of groups over a
fixed field $k$. A representation of a group we consider as a $2$-sorted
algebra. The first sort is a sort of elements of a\ group, the second sort
is a sort of vectors of some vector space over a fixed field $k$. As in
Example \ref{vector_space}, a multiplication of vectors by any scalar $%
\lambda \in k$ we consider as an unary operation. All operations over
vectors of the vector space have all arguments of the sort $2$ and give a
result of the sort $2$. All operations over elements of the group: the
multiplication; the operation that gives to each element its inverse and the
constant $1$ - have all arguments of the sort $1$ and give a result of the
sort $1$. An operation of action of elements of the group over vectors of
the vector space, which we denote by $\circ $, has a type $\tau _{\circ
}=\left( 1,2;2\right) $.

We denote by $\Theta $ some nondegenerate variety of representations of
groups over a fixed field $k$. By Proposition \ref{1_nondeg} the variety $%
\Theta ^{\left( 1\right) }$ is nondegenerate. So, $\Theta ^{\left( 1\right)
} $ is a nondegenerate variety of groups and, by Example \ref{gr}, $\Theta
^{\left( 1\right) }$ is an IBN-variety. Hence, by Proposition \ref{1_IBN}, $%
\Theta $ is an $1$-IBN-variety.

As a variety $\Lambda $ we consider the variety of trivial action, i.e., the
variety defined by identity%
\begin{equation*}
x^{\left( 1\right) }\circ x^{\left( 2\right) }=x^{\left( 2\right) }\text{.}
\end{equation*}%
This is identity (\ref{tr_act}) with $s\left( x^{\left( 2\right) }\right)
=x^{\left( 2\right) }$.

\begin{proposition}
\label{group_delta_2_nondeg}$\Delta ^{\left( 2\right) }$ is the variety of
all vector spaces over the field $k$.
\end{proposition}

\begin{proof}
In the proof of this proposition we denote elements $\delta _{\Xi ,X}\left(
x\right) $, $\delta _{\Lambda ,X}\left( x\right) $, $\delta _{\Theta
,X}\left( x\right) $, $\delta _{\Delta ,X}\left( x\right) $ and $x\in
X\subset X_{0}$ by same symbol $x$. This cannot cause confusion here.

We can prove by method of \cite[Section 3]{TsurkovSubvarLie} that for every
variety $\Xi $ of representations Lie algebras over the field $k$ the free
algebra $F_{\Xi }\left( X\right) $ of this variety generated by set $%
X=X^{\left( 1\right) }\uplus X^{\left( 2\right) }\subset X_{0}$ has this
form:%
\begin{equation*}
\left( F_{\Xi }\left( X\right) \right) ^{\left( 1\right) }=G_{\Xi }\left(
X^{\left( 1\right) }\right) =G\left( X^{\left( 1\right) }\right) /H_{\Xi
}\left( X^{\left( 1\right) }\right) ,
\end{equation*}%
\begin{equation*}
\left( F_{\Xi }\left( X\right) \right) ^{\left( 2\right)
}=\bigoplus\limits_{x^{\left( 2\right) }\in X^{\left( 2\right) }}\left(
kG\left( X^{\left( 1\right) }\right) /I_{\Xi }\left( X^{\left( 1\right)
}\right) \right) x^{\left( 2\right) },
\end{equation*}%
where $G\left( X^{\left( 1\right) }\right) $ is a free group, generated by
the set $X^{\left( 1\right) }$, $H_{\Xi }\left( X^{\left( 1\right) }\right) $
is a fully invariant subgroup of this group, $kG\left( X^{\left( 1\right)
}\right) $ is a $k$-group algebra over the group $G\left( X^{\left( 1\right)
}\right) $, $\bigoplus\limits_{x^{\left( 2\right) }\in X^{\left( 2\right)
}}\left( kG\left( X^{\left( 1\right) }\right) /I_{\Xi }\left( X^{\left(
1\right) }\right) \right) x^{\left( 2\right) }$ is a direct sum of $%
\left\vert X^{\left( 2\right) }\right\vert $ copies of $kG\left( X^{\left(
1\right) }\right) $ cyclic module, $I_{\Xi }\left( X^{\left( 1\right)
}\right) $ is a two-sided ideal of algebra $kG\left( X^{\left( 1\right)
}\right) $, which is invariant under all endomorphisms of $kG\left(
X^{\left( 1\right) }\right) $, defined by endomorphisms of $G\left(
X^{\left( 1\right) }\right) $. See also \cite[Section 0.2]{Vovsi}.

By \cite[Section 0.2, Example 1]{Vovsi}, $I_{\Lambda }\left( X^{\left(
1\right) }\right) =\mathfrak{Aug}\left( X^{\left( 1\right) }\right) $ an
augmentation ideal of $kG\left( X^{\left( 1\right) }\right) $, i.e., kernel
of the homomorphism $\varepsilon :kG\left( X^{\left( 1\right) }\right)
\rightarrow k$, defined by homomorphism $G\left( X^{\left( 1\right) }\right)
\rightarrow \left\{ 1_{k}\right\} $. $I_{\Theta }\left( X^{\left( 1\right)
}\right) \subseteq \mathfrak{Aug}\left( X^{\left( 1\right) }\right) $ by 
\cite[Proposition 0.2.3]{Vovsi}. Therefore, as in Example \ref{Lie_rep}, $%
I_{\Delta }\left( X^{\left( 1\right) }\right) =\left\langle I_{\Theta
}\left( X^{\left( 1\right) }\right) ,I_{\Lambda }\left( X^{\left( 1\right)
}\right) \right\rangle =I_{\Lambda }\left( X^{\left( 1\right) }\right) =%
\mathfrak{Aug}\left( X^{\left( 1\right) }\right) $. Hence%
\begin{equation*}
\left( F_{\Delta }\left( X\right) \right) ^{\left( 2\right)
}=\bigoplus\limits_{x^{\left( 2\right) }\in X^{\left( 2\right) }}\left(
kG\left( X^{\left( 1\right) }\right) /\mathfrak{Aug}\left( X^{\left(
1\right) }\right) \right) x^{\left( 2\right) }\cong
\bigoplus\limits_{x^{\left( 2\right) }\in X^{\left( 2\right) }}kx^{\left(
2\right) }
\end{equation*}%
is a vector space over $k$ which has a dimension $\left\vert X^{\left(
2\right) }\right\vert $. Now we complete the proof as in Example \ref%
{Lie_rep}.
\end{proof}

We have, as in Example \ref{Lie_rep}, that $\Delta $ is a $2$-nondegenerate $%
2$-IBN-variety and $\Theta $ is an IBN-variety.

\section{Acknowledgements}

I am grateful to the Professors A. Sivatski, E. Aladova, N. Cohen, A. Kuzmin
and to my student R. Barbosa Fernandes for fruitful discussions and
important remarks, which helped a lot in writing the article.

\end{document}